\documentclass[a4paper,12pt]{amsart}

  \usepackage{amssymb,amsthm}
\usepackage{graphicx,color}    
\usepackage{changes}
\usepackage{url} 

\newtheorem{theorem}{Theorem}[section]
\newtheorem{corollary}[theorem]{Corollary}
\newtheorem{proposition}[theorem]{Proposition}

\theoremstyle{definition}
\newtheorem{definition}[theorem]{Definition}
\newtheorem{example}[theorem]{Example}
\newtheorem{remark}[theorem]{Remark}

\def\r{\mathbb R}
\def\h{\mathbb H}
\def\l{\mathbb L}
\def\s{\mathbb S}

\begin{document}

\title{Extension of a problem of Euler in    Lorentz-Minkowski plane}
 
\author{Muhittin Evren Aydin}
\address{Department of Mathematics, Faculty of Science, Firat University, Elazig,  23200 Turkey}
\email{meaydin@firat.edu.tr}
\author{Rafael L\'opez}
\address{Departamento de Geometr\'{\i}a y Topolog\'{\i}a Universidad de Granada 18071 Granada, Spain}
\email{rcamino@ugr.es}
\begin{abstract}
In this paper we study curves  in Lorentz-Minkowski space $\l^{2}$ that are critical points of the moment of inertia with respect to the origin. This extends a problem posed by Euler in the Lorentzian setting.  We obtain explicit solutions for stationary curves in $\l^2$ distinguishing if the curve is spacelike or timelike. We also give a method to carry stationary spacelike curves into stationary timelike curves and vice versa via symmetries and inversions about the lightlike cone. Finally, we solve the problem of maximizing the energy among all spacelike curves joining two given points which are collinear with the origin.
\end{abstract}

\maketitle
\section{Introduction of the problem}

In 1744, Euler proposed the study of curves $\gamma$ in the Euclidean plane that minimize or maximize the energy
\begin{equation}\label{Eu1}
E_{\alpha}[\gamma]=\int_\gamma |\gamma(s)|^\alpha\, ds,
\end{equation}
where $\alpha\in\r$ is a parameter and $s$ is the arc-length parameter \cite{eu}.  When $\alpha=0$, the energy $E_0[\gamma]$ is the  length of $\gamma$. For  $\alpha=2$, the energy $E_2[\gamma]$ is the moment of inertia of $\gamma$ with respect to the origin of $\r^2$. Assuming that $\gamma$ does not contain the origin, the critical points of $E_\alpha$ are  characterized in terms of the curvature $\kappa$ of the curve, namely,  
\begin{equation}\label{Eu2}
\kappa =\alpha\frac{\langle N ,\gamma \rangle}{|\gamma|^2},
\end{equation}
where  $N$ is the unit normal vector of $\gamma$.     Recently, the authors together A. Bueno have extended this problem to the other two Riemannian space forms, the sphere $\mathbb{S}^2$ and the hyperbolic plane $\mathbb{H}^2$ \cite{abl}.
 
In this paper, we extend Euler's variational problem to the Lorentzian two-dimensional space form. The  Lorentz-Minkowski plane $\l^2$ is defined by  $\l^2=(\r^{2},\langle,\rangle)$ where $\langle,\rangle$ is the  Lorentzian  metric   $\langle,\rangle=dx^2-dy^2$ and   $(x,y)$ are the canonical coordinates of $\r^{2}$. Given two points $P_1, P_2\in\l^2$, the problem is to determine the curve $\gamma\colon I\to\l^2$ with minimum or maximum energy $E_\alpha$ among all curves joining $P_1$ and $P_2$.  

Important differences appear in the Lorentzian context. First, in order for the arc-length $ds$ to make sense  in \eqref{Eu1}, it is necessary to restrict the study to spacelike curves. In addition, since, for $\alpha=0$, the energy $E_0[\gamma]$ reduces to the length of the curve $\gamma$, it is natural in this context to find curves which maximize the energy $E_\alpha$, in contrast to the Euclidean case, where the problem concerns minimizers of the energy.

As in the Euclidean setting, we require that $0\not\in \gamma(I)$, $s\in I\subset \r$, to assure differentiability of the functional $E_\alpha$. However, another difference arises in the Lorentzian setting concerning the term $|\gamma|^\alpha$ in \eqref{Eu1}. In $\l^2$, the quantity $\langle \gamma,\gamma\rangle$ is not necessarily positive, because for a point $p=(x,y)\in\l^2$, we have $\langle p,p\rangle=x^2-y^2$, which may be positive, zero, or negative. Thus, we will take $|\gamma|^\alpha=|\langle\gamma,\gamma\rangle|^{\alpha/2}$. In consequence,  we need to assume that the curve $\gamma$ does not intersect the lightlike cone $\mathcal{C}=\{p\in\l^2\colon \langle p,p\rangle=0\}$.  

The compliment of $\mathcal{C}$ in $\l^2$ consists of four connected components, hence $\gamma$ must be included in one of these components. We set
$$\mathcal{C}^-=\{p\in\l^2\colon \langle p,p\rangle <0\},\qquad \mathcal{C}^+=\{p\in\l^2\colon \langle p,p\rangle >0\}.$$
Each of the sets $\mathcal{C}^-$ and $\mathcal{C}^+$ is formed by two components. Depending on whether $\gamma(I)$ is contained in $\mathcal{C}^-$ or in $\mathcal{C}^+$, we have $|\langle\gamma,\gamma\rangle|=-\langle\gamma,\gamma\rangle$ or $|\langle\gamma,\gamma\rangle|=\langle\gamma,\gamma\rangle$, respectively.   This makes  to distinguish in which set the curve is contained. 

Following the standard calculus of variations (see Sect. \ref{s2}), a spacelike curve $\gamma\colon I\subset\r\to\l^2\setminus\mathcal{C}$ is a critical point of the energy $E_\alpha$ if and only if 
\begin{equation}\label{ss}
\kappa= -\alpha \frac{ \langle N,\gamma\rangle}{|\langle \gamma,\gamma\rangle|}.
\end{equation}
Eq. \eqref{ss} leads to an interesting observation. Although the problem was initially formulated only for spacelike curves, the same equation also holds for timelike curves. This allows to introduce the following definition, which applies to both types of curves.
 
\begin{definition} \label{d1}
Let $\gamma\colon I\to\l^2$ be a spacelike or timelike curve in $\l^2$ satisfying $\gamma(I) \cap \mathcal{C} = \emptyset$. We call $\gamma$ an $\alpha$-stationary surface of $E_\alpha$ if  
\begin{equation}
\kappa=-\alpha \frac{ \langle N,\gamma\rangle}{|\langle \gamma,\gamma\rangle|}. \label{eq1}
\end{equation}
\end{definition}
 
If the value of $\alpha$ is not specified, such curves will simply be called stationary curves.

A particular emphasis in this paper is to highlight the differences from the Euclidean case and the specific features that appear in the Lorentzian setting. After the derivation of Eq. \eqref{eq1} for spacelike curves in Sect. 
\ref{s2}, we will show examples of stationary curves  with constant curvature, namely straight lines and circles. 
We also give an application of the maximum principle to control the value of $\alpha$ for $\alpha$-stationary spacelike curves. 
 In Sect. \ref{s3} we establish two types of  connections between $\alpha$-stationary curves. First,  by using the symmetry $(x,y)\mapsto (y,x)$ about the lightlike cone, we carry stationary spacelike curves contained in $\mathcal{C}^+$ (resp. $\mathcal{C}^-$) to stationary timelike curves contained in $\mathcal{C}^-$ (resp. $\mathcal{C}^+)$ (Thm. \ref{t31}). On the other hand, the inversion $p\mapsto \frac{p}{\langle p,p\rangle}$ gives a connection between  $\alpha$-stationary curves with different values of $\alpha$ (Thm. \ref{ti}).  Both results will be useful in the computation and derivation of properties of  stationary curves. In   Sect. \ref{s4}, we obtain the  explicit parametrizations of all stationary curves in $\l^2$.   Finally, in Sect. \ref{s5}, we provide a first approach to the problem of maximizing the energy $E_\alpha$, obtaining the classification of the maximizers in the case the two given points which are collinear with the origin.

\section{The Euler-Lagrange equation and examples}\label{s2}

In this section, we begin with the derivation of the Euler-Lagrange equation for the energy \eqref{Eu1}, next show examples of stationary curves, and finally apply the maximum principle to estimate the value of $\alpha$ for $\alpha$-stationary spacelike curves that are far away from the lightlike cone.  

The Euler-Lagrange equation follows from standard arguments in calculus of variations.  Let $\gamma\colon I\to\l^2$ be a spacelike curve, and assume that locally  $\gamma$ is given by $y=y(x)$, $x\in [a,b]$. The condition that $\gamma$ is spacelike implies that $y'(x)^2<1$ for all $x\in [a,b]$. As a first step, we consider the case that $\gamma$ is included in $\mathcal{C}^-$, that is,  $x^2-y(x)^2 <0$ for all $x\in [a,b]$. We take the unit normal vector   $N$ of $\gamma$ to be
\begin{equation}\label{n}
N=\frac{(y',1)}{\sqrt{1-y'^2}}.
\end{equation} 

The energy functional \eqref{Eu1} can be written as
\begin{equation}\label{ee}
E_\alpha[\gamma]=\int_a^b (y^2-x^2)^{\alpha/2}\sqrt{1-y'^2}\, dx:=\int_a^b  J(x,y,y').
\end{equation}
Applying the  Euler-Lagrange equation 
$\frac{\partial J}{\partial y}=\frac{d}{dx}\left(\frac{\partial J}{\partial y'}\right)$,
we obtain
\begin{equation}\label{eq0}
\alpha \frac{xy'-y}{\sqrt{1-y'^2}}=-(y^2-x^2)\frac{d}{dx}\left(\frac{y'}{\sqrt{1-y'^2}}\right).
\end{equation}
Since
$$\langle N,\gamma\rangle=\frac{xy'-y}{\sqrt{1-y'^2}},$$
and the second parenthesis in the right hand-side is the  curvature $\kappa$ of $\gamma$, Eq. \eqref{eq0} becomes
$$\kappa=\alpha\frac{\langle N,\gamma\rangle}{\langle \gamma,\gamma\rangle}.$$
This equation coincides with \eqref{eq1}. 

In case that $\gamma(I)$ is contained in $\mathcal{C}^+$, the arguments are similar. Now the energy changes to
$$E_\alpha[\gamma]=\int_a^b (x^2-y^2)^{\alpha/2}\sqrt{1-y'^2}\, dx,$$
obtaining that critical points of the energy are characterized by the equation
$$\kappa=-\alpha\frac{\langle N,\gamma\rangle}{\langle \gamma,\gamma\rangle}.$$
Again, this equation coincides with  \eqref{eq1}. 

Although the energy \eqref{Eu1} does not make sense for timelike curves, Eq. \eqref{eq1} still holds. For timelike curves, we     apply a similar argument. Indeed,  if the curve is $y=y(x)$, the timelike condition is $y'^2>1$. For this, we consider the same expression \eqref{ee} but replace the term $\sqrt{1-y'^2}$ with $\sqrt{y'^2-1}$. 

If $\gamma$ is contained in $\mathcal{C}^+$, the computation of the Euler-Lagrange equation yields  
$$\kappa=-\alpha\frac{\langle N,\gamma\rangle}{x^2-y^2}=-\alpha\frac{\langle N,\gamma\rangle}{\langle \gamma,\gamma\rangle}.$$
This formula coincides with \eqref{eq1}. In the case that $\gamma(I)\subset \mathcal{C}^-$, the same relation is obtained, except with the opposite sign on the right-hand side. Therefore, this confirms that the extension of the Def. \ref{d1} for timelike curves is valid.

\begin{remark} 
It is clear that linear isometries of $\l^2$, as well as dilations of $\l^2$, preserve the solutions of \eqref{eq1}.
\end{remark}

We present  some examples of stationary curves among the family of curves of $\l^2$ with constant curvature. Notice that by curvature we mean the (geodesic) curvature of a curve as submanifold of $\l^2$. If $\gamma$ is a non-degenerate curve in $\l^2$ parametrized by arc-length and $\nabla$ is the Levi-Civita connection on $\l^2$, the Weingarten formula is  $Av=-\nabla_v N$ where $A$ is the shape operator, $v$ is a tangent vector  and $N$ is the unit normal vector of $\gamma$. Then  the curvature is defined by $\kappa=-\epsilon\, \mbox{trace}(A)$, where   $\epsilon=\langle \gamma',\gamma'\rangle$. Then  
\begin{equation}\label{cu}
\begin{split}
\kappa&=-\epsilon\, \mbox{trace}(A)=-\epsilon(\langle A\gamma',\epsilon\gamma'\rangle= \langle \nabla_{\gamma'}N,\gamma'\rangle=\langle N ',\gamma'\rangle\\
&=-\langle N,\gamma'' \rangle.
\end{split}
\end{equation}
We now describe the two types of circles in $\l^2$ depending on its causal character. A hyperbolic circle in $\l^2$ with radius $r>0$ and center $p_0$ is defined as  one of the two components of
$$
\h^1(p_0;r)=\{ p \in \l^2 : \langle p-p_0, p-p_0 \rangle =-r^2\}.
$$
Each components is contained in a distinct component  of $\mathcal{C}^-$. This curve is spacelike with curvature $\kappa=\frac{1}{r}$ for the orientation $N(p)=(p-p_0)/r$.

Similarly, a pseudocircle of radius $r>0$ centered at $p_0$ is defined by one of the two components of
$$
\s^1_{1}(p_0;r)=\{ p \in \l^2 : \langle p-p_0, p-p_0 \rangle =r^2\}.
$$
Each component lies in one of the components of $\mathcal{C}^+$.  A pseudocircle is a timelike curve with  curvature  $\kappa=\frac{1}{r}$  for the orientation $N(p)=-(p-p_0)/r$. For both types of circles, when the center is the origin, we denote them simply by $\h^1(r)$ and $\s^1_1(r)$, respectively.

We provide some examples of stationary curves. 
\begin{example}  
\begin{enumerate}  
\item  All non-degenerate straight lines in $\l^2$ passing through the origin $(0,0)$ are $\alpha$-stationary curves, for any  value of $\alpha$.
\item    Hyperbolic  circles centered at the origin. They are $\alpha$-stationary curves for $\alpha=1$ and for all values of $r$. 
 
\item  Pseudocircles centered at the origin. They are $\alpha$-stationary curves for $\alpha=1$  for all values of $r$.   \end{enumerate}
 \end{example}

In the following result, we find all straight lines, hyperbolic circles and pseudocircles that are $\alpha$-stationary curves:  see   Fig. \ref{fig1}. 

\begin{proposition}\label{pr1}
\begin{enumerate}
\item Non-degenerate straight lines crossing $(0,0)$ are the  only straight lines that are $\alpha$-stationary curves. This holds for any value of $\alpha$.
\item The only hyperbolic circles that are $\alpha$-stationary curves are:
\  \begin{enumerate}
\item  Hyperbolic circles $\h^1(r)$ ($\alpha=1$) and the component of  $\h^1(p_0;r)$ with $\langle p_0,p_0\rangle=-r^2$ contained in $\mathcal{C}^-$ ($\alpha=2$). 
\item The component of  $\h^1(p_0;r)$ with $\langle p_0,p_0\rangle=-r^2$ contained in $\mathcal{C}^+$ ($\alpha=-2$). This component crosses $(0,0)$.  
\end{enumerate}
\item The only pseudocircles that are $\alpha$-stationary curves are:
\begin{enumerate}
\item  Pseudocircles $\s^1_{1}(r)$ ($\alpha=1$) and  the component of    $\s^1_{1}(p_0;r)$ with $\langle p_0,p_0\rangle=r^2$  ($\alpha=2$) contained in $\mathcal{C}^+$.  
\item The component of  $\s^1_{1}(p_0;r)$   with $\langle p_0,p_0\rangle=r^2$  ($\alpha=-2$) and contained in $ \mathcal{C}^-$. This component  crosses $(0,0)$.  
\end{enumerate} 
 \end{enumerate}
\end{proposition}

\begin{proof}
It is straightforward that the only straight lines that are stationary curves are those ones crossing the origin. 

\begin{enumerate}
\item We are looking for which hyperbolic circles are $\alpha$-stationary curves. Let $\h^1(p_0;r)$ be parametrized by $\gamma(s)=(x_0,y_0)+r(\sinh(s),\pm\cosh(s))$, where $p_0=(x_0,y_0)$. Suppose that  $\h^1(p_0;r)\subset\mathcal{C}^-$. Then \eqref{eq1} is 
$$(1-\alpha)\langle\gamma,\gamma\rangle+\alpha\langle\gamma,p_0\rangle=0,$$
which can be written as
$$-(\alpha -2) r x_0 \sinh (s)  \pm (\alpha -2) r y_0 \cosh (s)+x_0^2-y_0^2+(\alpha-1)r^2=0.$$
If $p_0=(0,0)$, then $\alpha=1$. Otherwise, we have $\alpha=2$ and   $\langle p_0,p_0\rangle=-r^2$. 
In such a case,  there exists a $s_0$ such that 
$$
x_0+r\sinh s_0=0, \quad y_0 \pm r\cosh s_0=0,
$$
namely, one of the components of $\gamma$ crosses $(0,0)$. Differentiating $\langle \gamma(s),\gamma(s)\rangle$ twice, we conclude $2\langle \gamma'(s_0),\gamma'(s_0)\rangle>0$, which means that $\langle \gamma(s),\gamma(s)\rangle >0$ around $s_0$, or equivalently $\gamma(s)\in \mathcal{C}^+$. This is not possible. The  other component is  included in $\mathcal{C}^-$, proving the result. 

Suppose now $\h^1(p_0;r)\subset\mathcal{C}^+$. Now Eq.  \eqref{eq1} is 
$$(1+\alpha)\langle\gamma,\gamma\rangle-\alpha\langle\gamma,p_0\rangle=0,$$
or equivalently, 
$$(\alpha +2) r x_0 \sinh (s)  \mp  (\alpha +2) r y_0 \cosh (s)+x_0^2-y_0^2-(\alpha+1)  r^2=0.$$
Notice that if $p_0=(0,0)$, then $\h^1(r)$ is included in $\mathcal{C}^-$, which it is not possible. Thus $p_0\not=(0,0)$ and consequently,  $\alpha=-2$ and   $\langle p_0,p_0\rangle=-r^2$. The component of $\h^1(p_0;r)$ that crosses $(0,0)$ is contained in $\mathcal{C}^+$ but the other is contained in $\mathcal{C}^-$.
 
 \item Let $\s^1_{1}(p_0;r)$ be a pseudocircle, which it is parametrized by $\gamma(s)=(x_0,y_0)+r(\pm \cosh (s),\sinh(s))$. Suppose that $\s^1_{1}(p_0;r)$ is contained in $\mathcal{C}^-$. Then \eqref{eq1} gives
 $$(1+\alpha)\langle\gamma,\gamma\rangle-\alpha\langle\gamma,p_0\rangle=0.$$
This equation can be written as
$$  \pm  (\alpha +2) r x_0 \cosh (s)-(\alpha +2) r y_0 \sinh (s)+x_0^2-y_0^2+(\alpha+1)r^2=0.$$
 The case $p_0=(0,0)$ is not possible because this pseudocircle  is included in $\mathcal{C}^-$. Thus $p_0\not=(0,0)$, hence $\alpha=-2$ and $\langle p_0,p_0\rangle=r^2$. Only one component is included in $\mathcal{C}^-$: the other one is included in $\mathcal{C}^+$, which must be discarded.
 
 Suppose   that  $\s^1_{1}(p_0;r)$ is contained in $\mathcal{C}^+$. Now \eqref{eq1} gives
 $$(1-\alpha)\langle\gamma,\gamma\rangle+\alpha\langle\gamma,p_0\rangle=0,$$
or equivalently, 
$$ \pm (2-\alpha) r x_0 \cosh (s)+(\alpha -2) r y_0 \sinh (s)+x_0^2-y_0^2-(\alpha-1)r^2=0.$$
If $p_0=(0,0)$, then we have $\alpha=1$. Otherwise, $p_0\not=(0,0)$, it follows that $\alpha=2$ and $\langle p_0,p_0\rangle=r^2$. Only one component is included in $\mathcal{C}^+$ which, in addition, crosses $(0,0)$.     \end{enumerate}
\end{proof}
 
 \begin{figure}[h]
\begin{center}
\includegraphics[width=.4\textwidth]{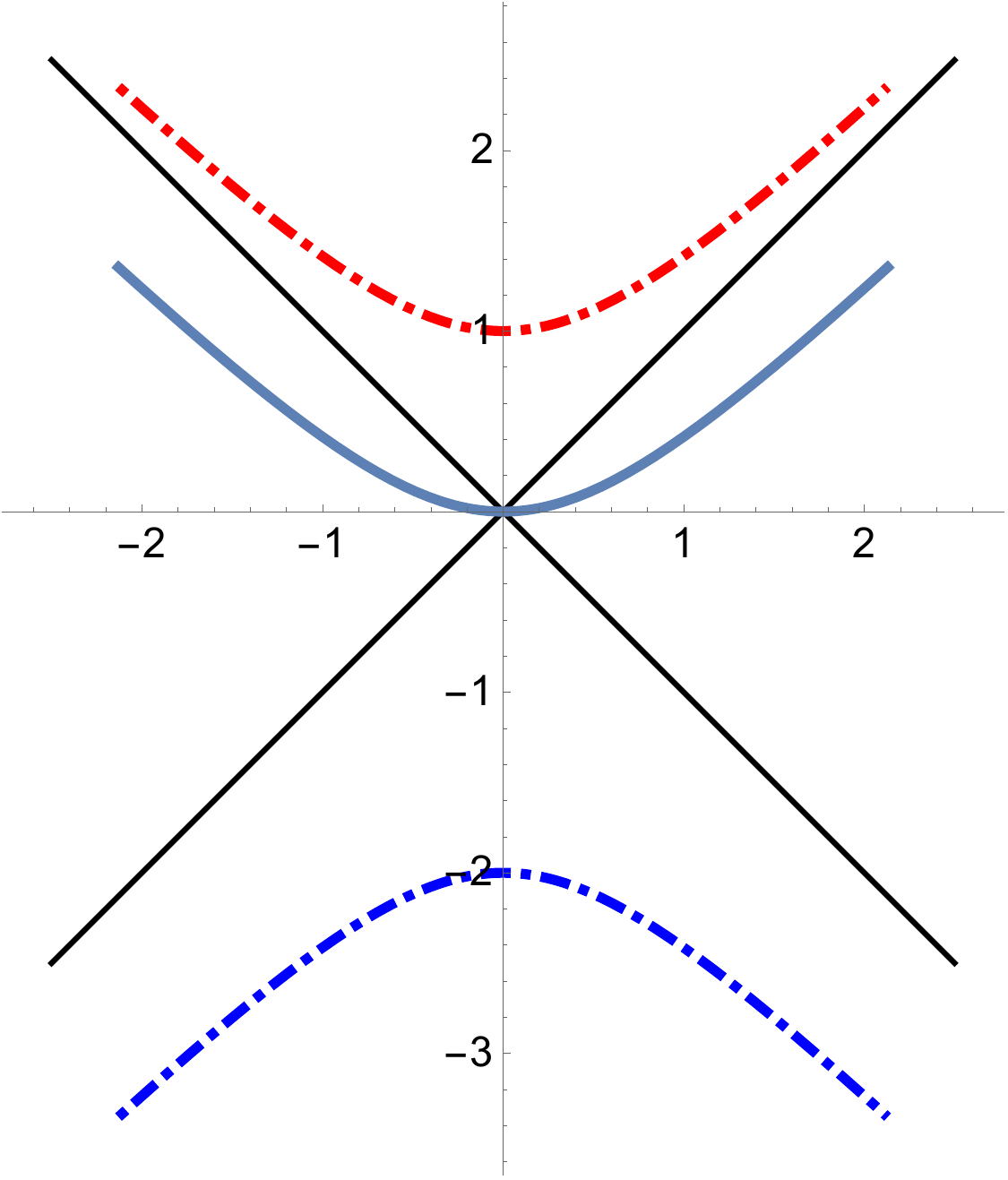}\qquad,\includegraphics[width=.5\textwidth]{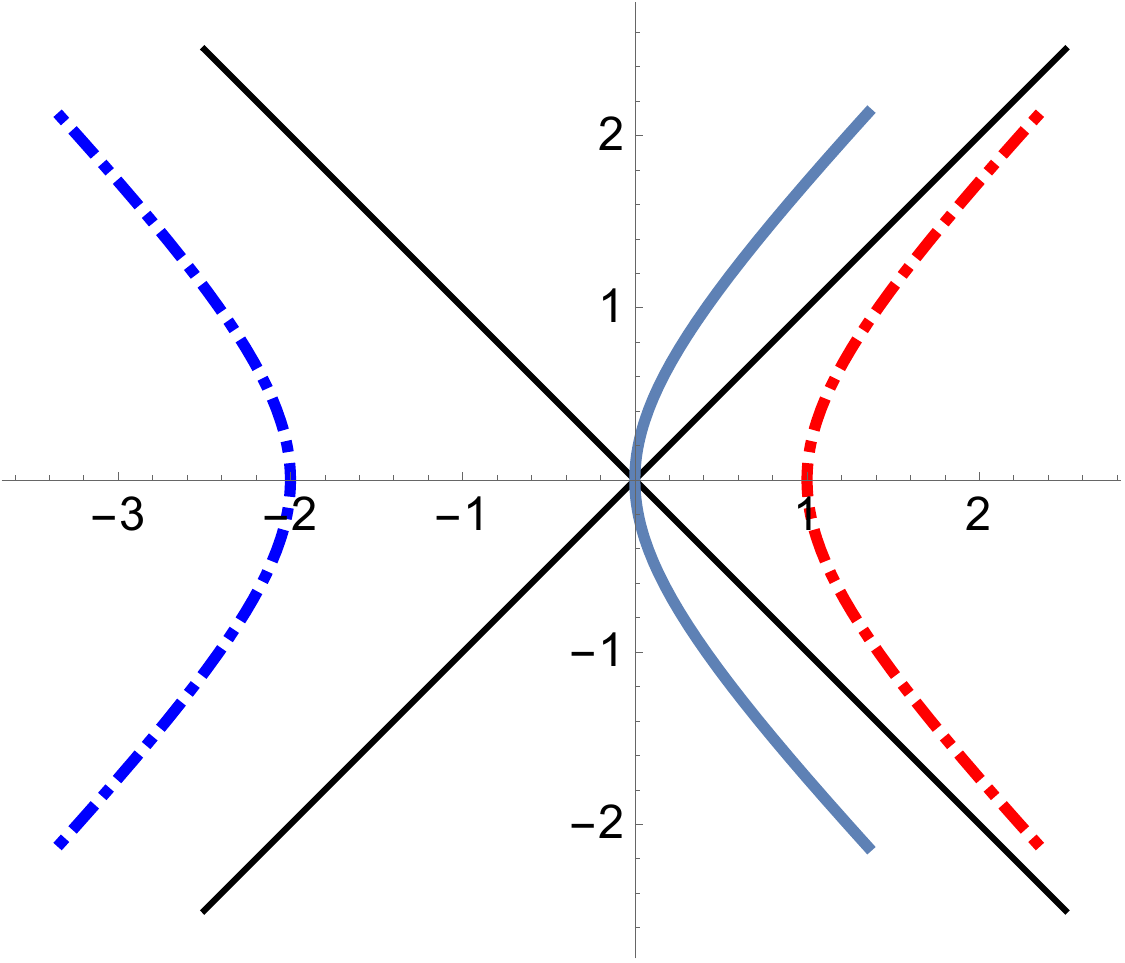}
\end{center}
\caption{Hyperbolic circles and pseudocircles that are stationary curves. Left: hyperbolic circles contained in $\mathcal{C}^{-}$ (dotted) and contained in $\mathcal{C}^+$ (solid). Right:  pseudocircles contained in $\mathcal{C}^{-}$ (solid) and contained in $\mathcal{C}^+$ (dotted).   }
\label{fig1}
\end{figure}

It deserves to point out a difference with the Euclidean case. In   $\r^2$, the only circles that are $\alpha$-stationary curves are those centered at $(0,0)$ ($\alpha=-1$)  and those crossing $(0,0)$ ($\alpha=-2$). In the Lorentzian plane, there are the analog circles (hyperbolic circles and pseudocircles) with suitable signs on $\alpha$. However, a new case appears in $\l^2$ which has not a counterpart in $\r^2$: there are hyperbolic circles and pseudocircles neither cross the origin and nor  centered at the origin: cases (2a) and (3a) in Prop. \ref{pr1}. The sign of $\alpha$ is that of the circles centered at the origin but its value differs.
 
We finish this section with an application of the maximum principle for elliptic equations. Eq. \eqref{eq1} is elliptic when the curve is spacelike  and the maximum principle can be applied: the argument is the same that in the Euclidean context (\cite[Prop. 2.5]{dl2}). First, for a curve $\gamma$, we define
$$\overline{\kappa}=\kappa+\alpha\frac{\langle N,\gamma\rangle}{|\langle \gamma,\gamma\rangle|}.$$
Given $\alpha\in\r$, the maximum principle asserts that if  two  spacelike curves $\gamma_i$, $i=1,2$, which are tangent at a common point $s=s_0$, and if $\gamma_1$ lies above $\gamma_2$ around $p$ according to the orientation by the (coincident) normal vectors at $s_0$, then $\overline{\kappa}_1(s_0)\geq \overline{\kappa}_2(s_0)$.  If, in addition,  $\overline{\kappa}_i$ are constant and $\overline{\kappa}_1=\overline{\kappa}_2$, then $\gamma_1=\gamma_2$ in a neighborhood of $s=s_0$. 
  
 Thanks to the maximum principle, it can be proved that in Euclidean plane, the only closed stationary curves are  circles centered at the origin ($\alpha=-1$). This result cannot have a counterpart in Lorentzian setting because   there are not closed spacelike curves. However, we will use hyperbolic circles $\h^1(r)$ to compare with $\alpha$-stationary spacelike curves obtaining information on the value of $\alpha$. For this, we give the following definition. A set $S\subset\l^2$ is said to be far away from the lightlike cone if there is $\delta>0$ such that 
  $$S\subset \mathcal{C}_\delta:=\{(x,y)\in\l^2\colon x^2-(y-\delta)^2<0\}.$$
  The notion of far away from the lightlike cone can be expressed in terms of causality of $\l^2$. It is equivalent to say that $S$ is contained in the future (resp. past) causal of the point $p_\delta:=(0,\delta)$, which is denoted in the literature by  $J^+(p_{  \delta})$ (resp. $J^{-}(p_\delta)$).   
    
  \begin{theorem}\label{t24}
   Let $\gamma$ be a properly immersed $\alpha$-stationary spacelike curve and  contained in $\mathcal{C}^-$. If $\gamma$ is far away   from the lightlike cone, then $\alpha>1$.
  \end{theorem}
  \begin{proof}
  Since $\gamma$ is connected, and after the linear isometry $(x,y)\mapsto (x,-y)$, we can assume that $\gamma$ is contained in the component of  $\mathcal{C}_\delta$ given by $\mathcal{C}_\delta^+=\{(x,y)\in \mathcal{C}_\delta\colon y\geq \delta\}$. The argument is similar if $\gamma$ is included in the other component. 
 Denote by $\h^1_+(r)$ the component of the hyperbolic circle $\h^1(r)$   contained in $\mathcal{C}^-\cap\{y\geq 0\}$. Notice that  $\h^1_+(r)$ is asymptotic to the two branches of $\mathcal{C}\cap\{y\geq 0\}$ and its vertex is $(0,r)$. Let $r$ be sufficiently small so that $\gamma(I)\cap\h^1_+(r)=\emptyset$. This is possible because $(0,r)\to (0,0)$ as $r\to 0$ and the boundary of   $ \mathcal{C}_\delta^+$ and the half-cone $\mathcal{C}\cap\{y\geq 0\}$ are at finite positive distance. In fact, the distance between $\gamma(I)$ and $\h^1_+(r)$ is positive and finite.

Let us increase the value of $r$,  $r\nearrow\infty$. Then there is a first number $r_0$ such that  $\h^1_+(r_0)$ and $\gamma(I)$ touch. Let  $s=s_0$ be a common point of $\gamma(I)\cap \h^1_+(r_0)$. By considering the unit normal vector $N(p)=p/r_0$ at $\h^1_+(r_0)$,   the set $\gamma(I)$ lies above $\h^1_+(r_0)$. For the parameter $\alpha$,  the maximum principle at the common contact point $s_0$ gives
  $$0=\overline{\kappa}_\gamma(s_0)\geq \overline{\kappa}_{\h^1_+(r_0)}(s_0)=\frac{1}{r}-\alpha\frac{1/r_0}{r_0^2}=\frac{1-\alpha}{r_0},$$
which implies $\alpha\geq 1$. However, the case $\alpha=1$ cannot occur. In such a situation, both curves $\gamma$ and  $\h^1_+(r_0)$ have the same constant value for $\overline{\kappa}$ and $\overline{\kappa}=0$. The maximum principle would imply that $\gamma$ coincides with the hyperbolic circle $\h^1_+(r_0)$. This is not possible because $\h^1_+(r_0)$ is not far away from the lightlike cone. 
  \end{proof}

\begin{remark} Notice that in Prop. \ref{pr1}, (2a), the component of $\h^1(p_0;r)$, with $\langle p_0,p_0\rangle=-r^2$ and contained in $\mathcal{C}^{-}$, is far away from the lightlike cone.  The value of $\alpha$ for this curve is $\alpha=2$, which is consistent with Thm. \ref{t24}. To check that this component is far away from the lightlike cone, we can assume without loss of generality that the curve is contained in $\mathcal{C}^{-}\cap \{y>0\}$ and $p_0=r(\sinh(t),\cosh(t))$ with $t>0$. Then the curve is contained in $C_\delta$ with $\delta=re^{-t}$.
\end{remark}
\section{Two relationships between stationary    curves}\label{s3}

Prop. \ref{pr1} provides similar statements between stationary (spacelike and timelike) circles: compare (2a) with (3a) and (2b) with (3b) respectively. This can be obtained from a more general viewpoint as we now show. 

Define the map 
$$\Phi\colon\l^2\to\l^2,\quad \Phi(x,y)=(y,x).$$
This map is an involution, that is,   $\Phi^2=1_{\l^2}$.   It is clear that $\Phi$ maps spacelike (timelike) vectors to timelike (spacelike) vectors. Moreover,  $\langle\Phi(p),\Phi(p)\rangle=-\langle p,p\rangle$. Consequently, $\Phi$ carries   spacelike curves  contained in $\mathcal{C}^+$ (resp. $\mathcal{C}^-$) to timelike curves contained in $\mathcal{C}^-$ (resp. $\mathcal{C}^+$). The reverse process also holds: timelike curves are carried to spacelike curves. Thus, the map $\Phi$ has a well-defined behaviour for stationary curves.

\begin{theorem}\label{t31}
 Let $\gamma$ be an $\alpha$-stationary spacelike curve contained in $\mathcal{C}^-$ (resp. $\mathcal{C}^+$). Then   $\widetilde{\gamma}=\Phi\circ\gamma$ is an $\alpha$-stationary timelike curve contained in $\mathcal{C}^+$ (resp. $\mathcal{C}^-$).   
\end{theorem} 
 
\begin{proof}
Let $\gamma$ be an $\alpha$-stationary spacelike curve contained in $\mathcal{C}^-$, parametrized by arc-length. Then we have the relations
$$\widetilde{\gamma}=\Phi\circ\gamma,\quad \widetilde{N}=\Phi\circ N.$$
It follows $|\langle\gamma,\gamma\rangle|=|\langle\widetilde{\gamma},\widetilde{\gamma}\rangle|$. Due to $\Phi^2=1_{\l^2}$, we derive
$$\langle N,\gamma\rangle=\langle\Phi\widetilde{N},\Phi\widetilde{\gamma}\rangle=-\langle\widetilde{N},\widetilde{\gamma}\rangle.$$
By using \eqref{cu}, we compute
$$\widetilde{\kappa}=-\langle\widetilde{\gamma''},\widetilde{N}\rangle=-\langle\Phi\circ\gamma'',\Phi \circ N\rangle=\langle\gamma'',N\rangle=-\kappa.$$
Finally, by \eqref{eq1}, we have
$$\widetilde{\kappa}=- \kappa=\alpha\frac{\langle N,\gamma\rangle}{|\langle\gamma,\gamma\rangle|}=-\alpha\frac{\langle \widetilde{N},\widetilde{\gamma}\rangle}{|\langle\widetilde{\gamma},\widetilde{\gamma}\rangle|}.$$
This proves that $\widetilde{\gamma}$ is an $\alpha$-stationary timelike curve contained in $\mathcal{C}^+$. The argument is similar when $\gamma$ is an $\alpha$-stationary spacelike curve contained in   $\mathcal{C}^+$.
\end{proof}
 
From this result, the statements of Proposition \ref{pr1} follow as a consequence, see also Fig. \ref{fig1}.
 
Another interesting map in $\l^2$ for $\alpha$-stationary curves is the inversion with respect to the lightlike cone, defined by
$$\Psi\colon\l^2\setminus\mathcal{C}\to\l^2,\quad \Psi(p)=\frac{p}{\langle p,p\rangle}.$$
Notice that $\Psi$ preserves the components of both $\mathcal{C}^+$ and $\mathcal{C}^-$. 
We prove that $\Psi$ maps $\alpha$-stationary curves to stationary curves, but with different value of $\alpha$. 

First, we introduce hyperbolic polar coordinates, which are the Lorentzian analogue of the polar coordinates in the Euclidean plane. Given a point $p=(x,y)\in\l^2$ with $\langle p,p\rangle\not=0$,   there exist $\rho>0$ and $\varphi\in\r$ such that 
\begin{equation*}
\begin{split}
p=\rho(\sinh\varphi,\pm\cosh\varphi) &\quad \mbox{if $p\in\mathcal{C}^-$}\\
p=\rho(\pm\cosh\varphi,\sinh\varphi) &\quad \mbox{if $p\in\mathcal{C}^+$.}\\
\end{split}
\end{equation*}
The sign $\pm$ appears because each of the sets $\mathcal{C}^-$ and $\mathcal{C}^+$ is formed by two components. Applying the linear symmetries of $\l^2$ given by  $(x,y)\mapsto (x,-y)$ and $(x,y)\mapsto (-x,y)$, we can restrict to the choice of the positive sign in the hyperbolic polar coordinates. 

Let $\gamma=\gamma(s)=(x(s),y(s))$ be a  spacelike curve   included in $\mathcal{C}^-$. Let $\rho=\rho(s)$ be a smooth function such that  
$$\gamma(s)=\rho(s)(\sinh(s), \cosh(s)).$$
The curvature $\kappa$ is given in terms of $\rho$ by 
$$\kappa=\frac{\rho  \left(\rho ''+\rho \right)-2 \rho '^2}{\left(\rho ^2-\rho '^2\right)^{3/2}},$$
where the unit normal vector is 
$$N=\frac{1}{\sqrt{\rho^2-\rho'^2}}(\sinh(s)\rho+\cosh(s)\rho',\cosh(s)\rho+\sinh(s)\rho').$$
The spacelike condition is $\rho^2-\rho'^2>0$. Then Eq. \eqref{eq1} is
\begin{equation}\label{k1}
\rho  \rho ''+(\alpha -2) \rho '^2+(1-\alpha ) \rho ^2=0.
\end{equation}
If $\gamma$ is contained in $\mathcal{C}^+$, the spacelike condition is $\rho'^2-\rho^2>0$, and Eq. \eqref{eq1} is
\begin{equation}\label{k2}
\rho\rho''-(\alpha+2)\rho'^2+(1+\alpha)\rho^2=0.
\end{equation}
If $\gamma$ is a timelike curve, then Eq. \eqref{eq1} coincides with \eqref{k1} when $\gamma$ is contained in $\mathcal{C}^+$ and with \eqref{k2} when $\gamma$ is contained in $\mathcal{C}^-$.

Once we have the equations for the $\alpha$-stationary curves, we now state the following result about the inversion with respect to the lightlike cone (see \cite{lo1} for the Euclidean counterpart).

\begin{theorem}\label{ti}
Let $\gamma$ be a non-degenerate curve contained in $\l^2 \setminus \mathcal{C}$, and let $\widetilde{\gamma}$ denote its inverse, $\widetilde{\gamma}=\Psi\circ\gamma$. Then $\widetilde{\gamma}$ and $\gamma$ have the same causal character. Moreover:
\begin{enumerate}
\item $\gamma$ is an $\alpha$-stationary curve contained in $\mathcal{C}^-$ if and only if $\widetilde{\gamma}$ is an $(2-\alpha)$-stationary curve contained in $\mathcal{C}^-$.
\item $\gamma$ is an $\alpha$-stationary curve contained in $\mathcal{C}^+$ if and only if $\widetilde{\gamma}$ is an $(-2-\alpha)$-stationary curve contained in $\mathcal{C}^+$.
\end{enumerate}
\end{theorem}

\begin{proof} 
By Theorem \ref{t31}, it suffices to consider spacelike curves. Suppose that $\gamma$ is contained in $\mathcal{C}^-$. In hyperbolic polar coordinates, the inverse curve $\widetilde{\gamma}$ can be written as $\widetilde{\gamma}(s)=\frac{1}{\rho(s)}(\sinh(s),\cosh(s))$. Thus $\widetilde{\rho}=1/\rho$. In terms of $\widetilde{\rho}$, Eq. \eqref{k1} is equivalent to
$$ \widetilde{\rho}\widetilde{\rho}''-\alpha \widetilde{\rho}'^2+(\alpha-1)\widetilde{\rho}^2=0.$$
Since $\widetilde{\gamma}$ is included in $\mathcal{C}^-$, a comparison with Eq. \eqref{k1} yields that $\widetilde{\gamma}$ is an $(2-\alpha)$-stationary spacelike curve. The argument is similar when $\gamma$ is contained in $\mathcal{C}^+$.
\end{proof}

The inversion $\Psi$ keeps invariant, as set,  the family of straight lines, hyperbolic circles and pseudocircles. We show some interesting examples of  applications of Thm. \ref{ti}. First, notice that any non-degenerate straight line of $\l^2$ is a curve of $\kappa=0$. Thus this curve satisfies Eq. \eqref{eq1} for $\alpha=0$ and we can apply Thm. \ref{ti}.

\begin{enumerate}
\item The components of  $\h^1(r)$ are contained in $\mathcal{C}^-$ and they  are interchanged by $\Psi$. The value of $\alpha $ ($\alpha=1$) is also preserved by Thm. \ref{ti}. 
\item Similarly, the pseudocircles $\s^1_1(r)$ are preserved via $\Psi$. They are $1$-stationary curves.
 \item Consider the straight line $\gamma$ of equation $y=1$, where $\gamma \cap \mathcal{C}=\emptyset$. Then $\gamma$ is spacelike and it has   components in $\mathcal{C}^+$ and in $\mathcal{C}^-$. Its image by $\Psi$ has three components. The component contained in  $\mathcal{C}^-$ goes to the hyperbolic circle $\h^1(1)$ contained in $\mathcal{C}^-\cap\{y<0\}$ and it is a $2$-stationary curve: see (2a) of Prop. \ref{pr1}.  The components of $\gamma$ contained in $\mathcal{C}^+$ go, via $\Psi$, to $(-2)$-stationary  curves in $\mathcal{C}^+$. They are the hyperbolic circles of (2b) of Prop. \ref{pr1}.
\item Let $\gamma$ be the vertical line of equation $x=1$, where $\gamma \cap \mathcal{C}=\emptyset$. This curve is timelike and stationary for $\alpha=0$. Also, $\gamma$ has components in $\mathcal{C}^-$ and in $\mathcal{C}^+$.   The  image of the component of  $\gamma$ in $\mathcal{C}^+$ via $\Psi$ is the pseudocircle of (3a) in Prop. \ref{pr1}, which is a  $(-2)$-stationary curve. The image of the components of $\gamma$ contained in $\mathcal{C}^-$ via $\Psi$ are  pieces of pseudocircles which are $(-2)$-stationary curves: see (3a) in Prop. \ref{pr1}.
\end{enumerate}

\section{Parametrizations of stationary   curves}\label{s4}

In this section, we obtain explicit parametrizations of all stationary (spacelike and timelike) curves of $\l^2$.
 We begin with obtaining all stationary spacelike curves contained in $\mathcal{C}^-$: see Fig. \ref{fig2}.

\begin{theorem} \label{pr41}
Up to a linear isometry and a dilation of $\l^2$, the   $\alpha$-stationary spacelike curves contained in $\mathcal{C}^-$ are parametrized by $\gamma(s)=\rho(s)(\sinh(s),\cosh(s))$, where 
\begin{equation}\label{so}
\rho(s)=\left\{\begin{array}{ll}
 \cosh((\alpha-1) s)^\frac{1}{\alpha-1},&\alpha\not=1\\
 e^{c s},&\alpha=1, c^2<1.
 \end{array}
 \right.
\end{equation}
 Moreover, 
\begin{enumerate}
\item If $\alpha> 1$, then  the two branches of  $\gamma$ are asymptotic to a vertical translation of the lightlike cone $\mathcal{C}$.  
\item If $\alpha=1$, then $\gamma$ is a hyperbolic circle $\h^1(r)$ centered at the origin or the two branches of $\gamma$ are asymptotic to $\mathcal{C}$. 
\item If $0<\alpha<1$, then $\gamma$ meets tangentially $\mathcal{C}$ at two points.  
\item If $\alpha<0$, then $\gamma$ intersects orthogonally $\mathcal{C}$ at two points.
\end{enumerate}
\end{theorem}
 
\begin{proof}
Let $\gamma=\gamma(s)=(\gamma_1(s),\gamma_2(s))$ be an $\alpha$-stationary spacelike curve   included in $\mathcal{C}^-$. Let $\rho=\rho(s)$ be a smooth function such that  
$$\gamma(s)=\rho(s)(\sinh(s), \cosh(s)).$$
We know that $\rho$ satisfies \eqref{k1} with  $\rho^2-\rho'^2>0$. This equation   can be solved,  obtaining
$$\rho(s)= \left\{\begin{array}{ll}
c_2 \left(e^{(\alpha-1)s}+c_1 e^{-(\alpha-1)s}\right)^\frac{1}{\alpha-1},&\alpha\not=1\\
c_2e^{c_1 s},&\alpha=1, c_1^2<1.
\end{array}\right.
$$
where $c_1,c_2\in\r$, $c_2\not=0$. The spacelike condition on $\gamma$ implies  $c_1>0$ if $\alpha\not=1$ and $c_1^2<1$ if $\alpha=1$.  

We separate the case $\alpha\not=1$. Consider the rigid motion $R_t(x,y)= (\cosh(t) x+\sinh(t) y,\sinh(t)x+y\cosh(t))$, where $t=-\frac{\log(c_1)}{\alpha-1}$. Then 
$$R_t\cdot \gamma(s)=\gamma(s+t)=(2\sqrt{c_1}\cosh((\alpha-1)\theta))^\frac{1}{\alpha-1}(\sinh\theta,\cosh\theta),$$
where $\theta=s+t$. After a dilation, we obtain the parametrization given in \eqref{so}. Let 
$$\rho(s)=(\cosh((\alpha-1)s))^\frac{1}{\alpha-1}.$$
As a consequence, the curve $\gamma$ is symmetric about the $y$-axis.
\begin{enumerate}
\item Case $\alpha>1$. It is clear that $\lim_{s\to\infty}\gamma(s)=\infty$. In order to prove that it is asymptotic to a lightlike cone, we have
$$\lim_{s\to\pm\infty}\frac{\gamma_2'(s)}{\gamma_1'(s)}=\pm 1,$$
which gives the slope of the asymptote. Moreover, 
$$\lim_{s\to\infty}(\gamma_2(s)-\gamma_1(s))=2^\frac{1}{1-\alpha}.$$
This implies that the lines of equation $y=\pm  x+2^\frac{1}{1-\alpha}$ are asymptotic to $\gamma$. Thus $\gamma$ is asymptotic to $(0,2^{\frac{1}{1-\alpha}})+\mathcal{C}$ which is a vertical translation of $\mathcal{C}$.
\item Case $0<\alpha<1$. A computation gives 
$$\lim_{s\to\pm\infty}\gamma(s)=(\pm 2^\frac{\alpha}{1-\alpha},2^\frac{\alpha}{1-\alpha})$$
and
$$\lim_{s\to\pm\infty}\frac{\gamma_2'(s)}{\gamma_1'(s)}=\pm 1.$$
This proves that the intersection of $\gamma$ with  $\mathcal{C}$ is tangential.
\item Case $\alpha<0$. We have 
$$\lim_{s\to\pm\infty}\gamma(s)=(\pm 2^\frac{\alpha}{1-\alpha},2^\frac{\alpha}{1-\alpha}),$$
proving $\gamma$ tends to $\mathcal{C}$ at two points. Since we also have
$$\lim_{s\to\pm\infty}\frac{\gamma_2'(s)}{\gamma_1'(s)}=\pm 1,$$
the intersection (at the limit) between $\gamma$ and $\mathcal{C}$ is orthogonal.
\end{enumerate}
Suppose now $\alpha=1$. If $c_1=0$, then   $\gamma$ is the hyperbolic circle $\h^1(-c_2^2)$.  If $c_1\not=0$, then $\lim_{s\to\pm\infty}\frac{\gamma_2'(s)}{\gamma_1'(s)}=\pm 1$. However, now we have $\lim_{s\to\pm\infty}\gamma_2(s)-(\pm)\gamma_1(s)=0$. This proves that the two branches of $\gamma$ are asymptotic to $\mathcal{C}$.

\end{proof}

\begin{figure}[h]
\begin{center}
\includegraphics[width=.3\textwidth]{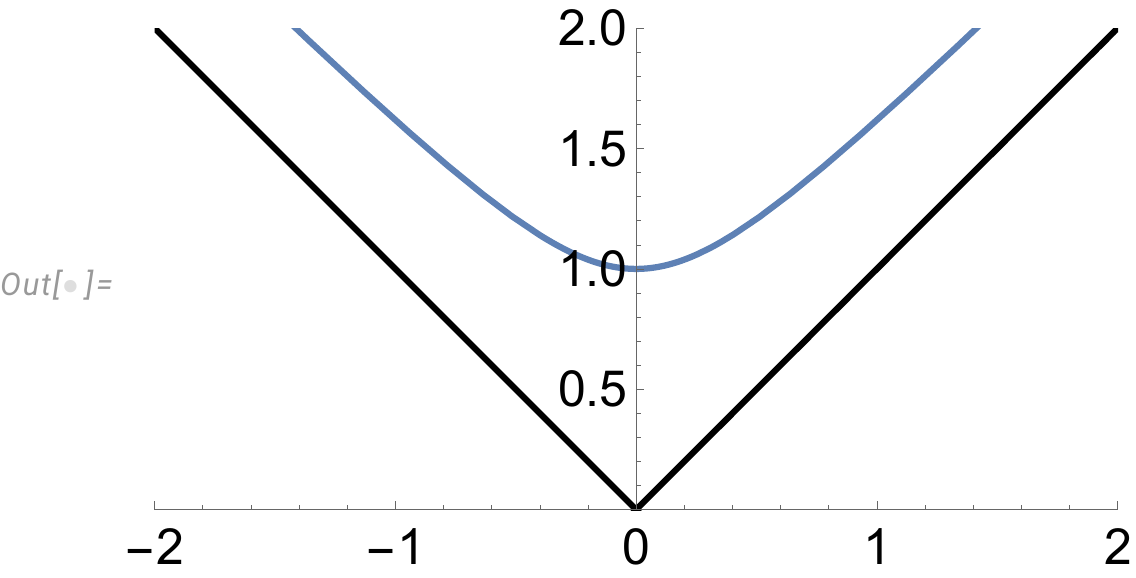}\,\includegraphics[width=.3\textwidth]{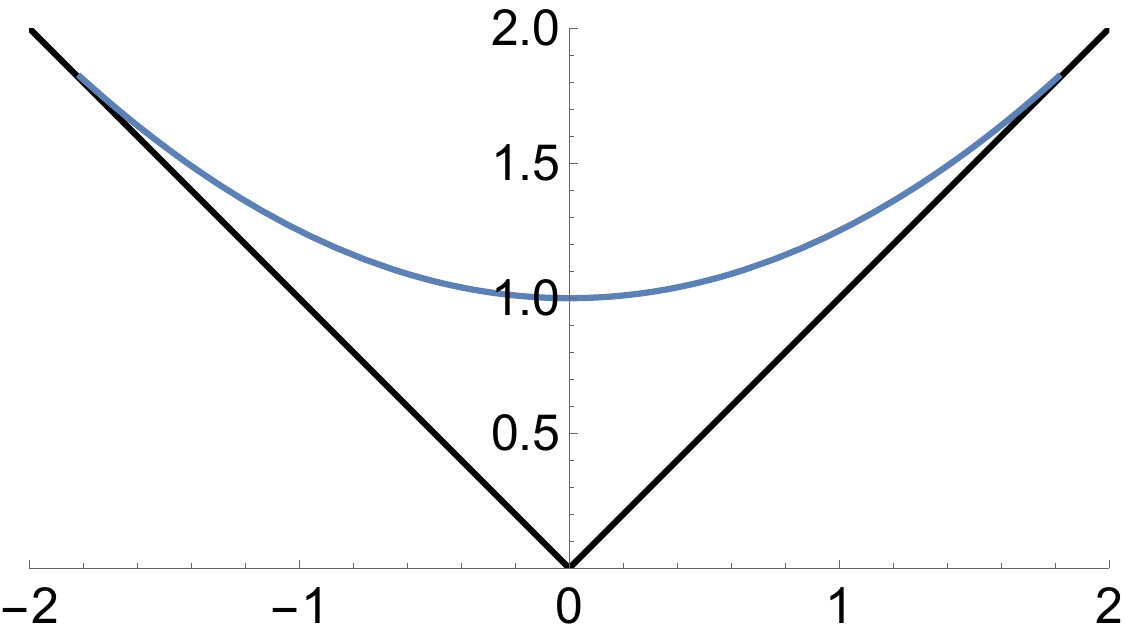}\,\includegraphics[width=.3\textwidth]{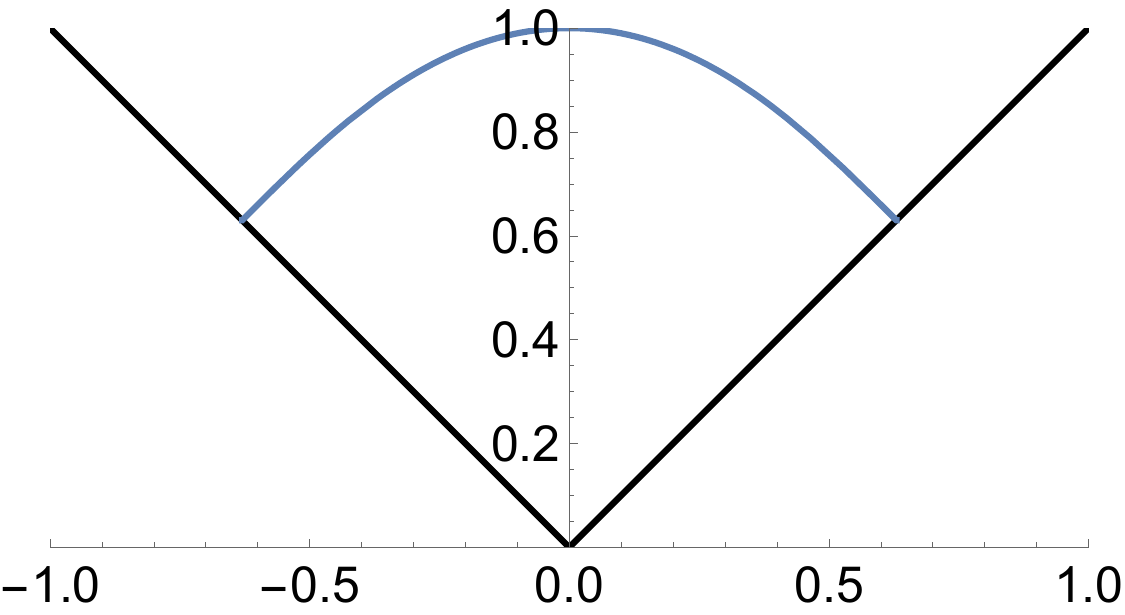}
\end{center}
\caption{  $\alpha$-stationary spacelike curves contained in $\mathcal{C}^-$: $\alpha=2$ (left), $\alpha=\frac12$ (middle) and $\alpha=-2$ (right).  }
\label{fig2}
\end{figure}

\begin{remark} 
We point out that for $\alpha>1$, the $\alpha$-stationary spacelike curves stated in Prop. \ref{pr41} are far away from the lightlike cone, consistent with Thm. \ref{t24}.
\end{remark}

The next case is to consider stationary spacelike curves contained in $\mathcal{C}^+$: see Fig. \ref{fig3}.

\begin{theorem} \label{pr42}
Up to a linear isometry and a dilation of $\l^2$, the   $\alpha$-stationary spacelike curves contained in $\mathcal{C}^+$ are parametrized   by $\gamma(s)=\rho(s)(\cosh(s),\sinh(s))$, where
\begin{equation}\label{so122}
\rho(s)=\left\{\begin{array}{ll}
 \sinh((\alpha+1) s)^\frac{-1}{\alpha+1}, &\alpha\not=-1\\
 e^{c s},&\alpha=-1, c^2>1.
 \end{array}\right.
\end{equation}
 Moreover, 
\begin{enumerate}
\item If $\alpha>0$, then a branch of  $\gamma$ touches $\mathcal{C}$  orthogonally and the other one is asymptotic to the $x$-axis. 

\item If $-1<\alpha<0$, then a branch of $\gamma$ goes to $\mathcal{C}$ tangentially.
\item If $\alpha=-1$, then a branch of $\gamma$ goes to $(0,0)$  and tangentially to  $\mathcal{C}$.  
\item If $\alpha<-1$, then a branch of $\gamma$ goes to $(0,0)$ tangentially to the $x$-axis and the other branch is asymptotic to a translation of the lightlike cone $\mathcal{C}$. 

\end{enumerate}
\end{theorem}

\begin{proof}
Since the curve is contained in $\mathcal{C}^+$, in hyperbolic polar coordinates, we have $\gamma(s)=\rho(s)(\cosh(s),\sinh(s))$. Then  $\gamma$ satisfies \eqref{k2} with  $\rho'^2-\rho^2>0$.   The solution is 
\begin{equation}\label{so12}
\rho(s)= \left\{\begin{array}{ll}
c_2 \left(e^{(\alpha+1)s}+c_1 e^{-(\alpha+1)s}\right)^\frac{-1}{\alpha+1},&\alpha\not=-1, c_1<0\\
c_2e^{c_1 s},&\alpha=-1, c_1^2>1,
\end{array}\right.
\end{equation}
where $c_2\in\r$. 

Suppose $\alpha\not=-1$.  A similar argument to that of Thm. \ref{pr41} concludes that, after  a rigid motion  and a dilation of $\l^2$, the function $\rho$ is given by \eqref{so122}.  The curve $\gamma$ is defined provided $(1+\alpha)s>0$.
 
If $\alpha>0$, the domain of   $\gamma$ is $(0,\infty)$ with
\begin{equation}\label{ll}
\lim_{s\to\infty}\gamma(s)=(2^\frac{\alpha}{\alpha+1},2^\frac{\alpha}{\alpha+1}),\quad \lim_{s\to 0}\gamma(s)=(\infty,0).
\end{equation}
The first limit in \eqref{ll} proves that $\gamma$ tends to $\mathcal{C}$. Since $\lim_{s\to\infty}\frac{\gamma_2'(s)}{\gamma_1'(s)}=-1$, the intersection is orthogonal. The second limit in \eqref{ll} proves that $\gamma$ is asymptotic to the $x$-axis.

Suppose $-1<\alpha<0$. Again the domain of $\gamma$ is $(0,\infty)$  with 
$$\lim_{s\to\infty}\gamma(s)=(2^\frac{-\alpha}{\alpha+1},2^\frac{-\alpha}{\alpha+1}),\quad \lim_{s\to 0}\gamma(s)=(\infty,0).$$
The difference with the item (1) is that $\lim_{s\to\infty}\frac{\gamma_2'(s)}{\gamma_1'(s)}=1$ which proves that $\gamma$ tends to $\mathcal{C}$ tangentially.
 
If $\alpha<-1$, then $\gamma$ is defined in $(-\infty,0)$. We have $\lim_{s\to 0}\gamma(s)=(0,0)$.  Since $\lim_{s\to\infty}\frac{\gamma_2'(s)}{\gamma_1'(s)}=0$, the curve $\gamma$ goes to the origin tangentially to the $x$-axis. On the other hand, 
$$\lim_{s\to-\infty}\frac{\gamma_2'(s)}{\gamma_1'(s)}=-1,\quad \lim_{s\to-\infty} (\gamma_2(s)- \gamma_1(s))=2^\frac{-1}{\alpha+1},$$
which proves that this branch of $\gamma$ is asymptotic to the line $y=-x+2^\frac{-1}{\alpha+1}$.

Finally, we consider the case $\alpha=-1$. We now have
$$\lim_{s\to-\infty} \gamma (s) =(0,0),\quad \lim_{s\to-\infty}\frac{\gamma_2'(s)}{\gamma_1'(s)}=-1.$$
This proves that $\gamma$ goes to $(0,0)$ tangentially to vertex of the lightlike cone $\mathcal{C}$.
\end{proof}

\begin{figure}[h]
\begin{center}
\includegraphics[width=.3\textwidth]{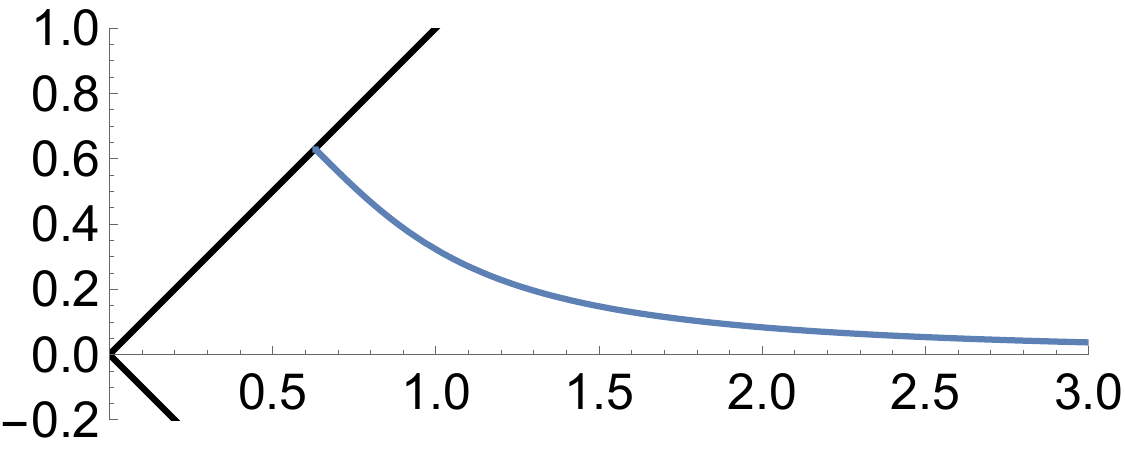}\,\includegraphics[width=.3\textwidth]{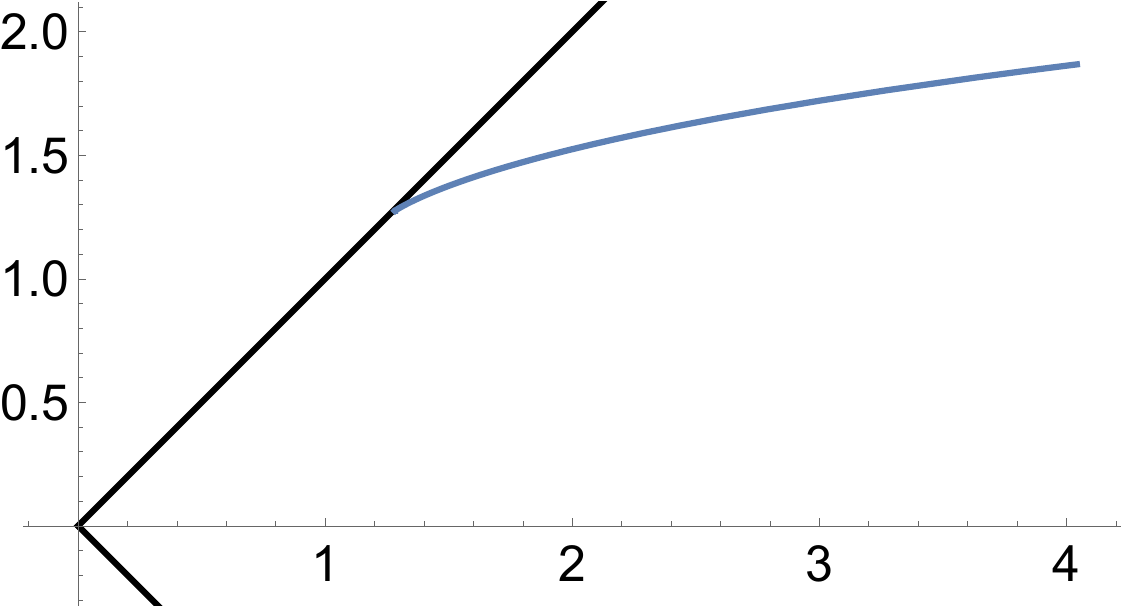}\,\includegraphics[width=.3\textwidth]{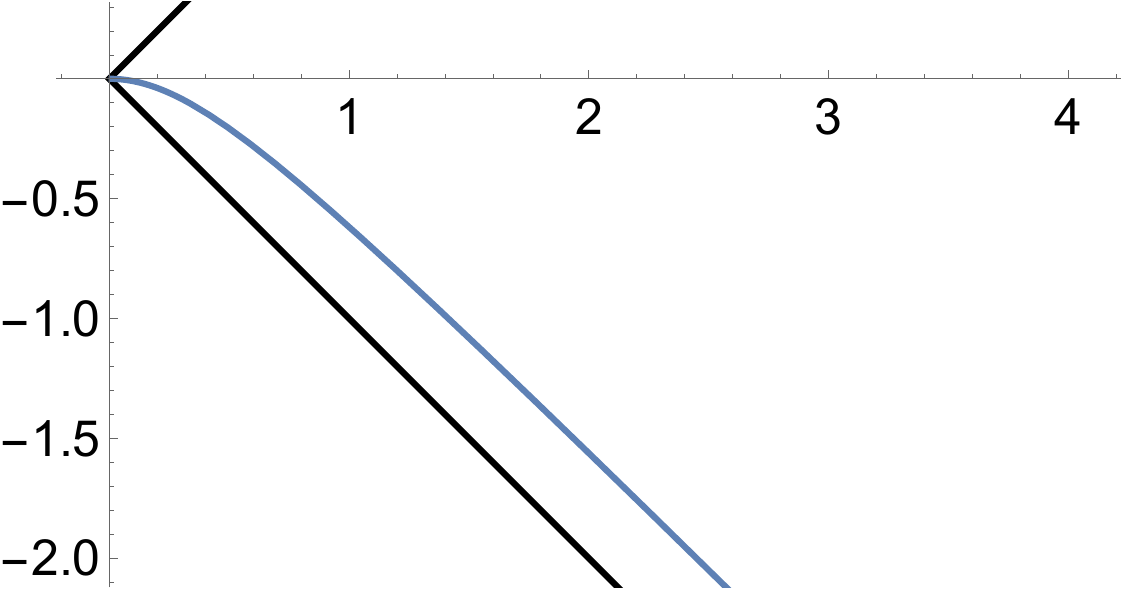}
\end{center}
\caption{  $\alpha$-stationary spacelike curves contained in $\mathcal{C}^+$: $\alpha=2$ (left), $\alpha=-\frac14$ (middle) and $\alpha=-2$ (right).  }
\label{fig3}
\end{figure}

We obtain the explicit parametrizations of  stationary timelike curves of $\l^2$ by using  Thm. \ref{t31}. Thm. \ref{pr41} which  give the stationary timelike curves contained in $\mathcal{C}^+$: see Fig. \ref{fig4}

\begin{theorem} \label{prt-}
Up to a linear isometry and a dilation of $\l^2$, the  $\alpha$-stationary timelike curves contained in $\mathcal{C}^+$ are parametrized  by $\gamma(s)=\rho(s)(\cosh(s),\sinh(s))$, where  
\begin{equation}\label{sot2}
\rho(s)=\left\{\begin{array}{ll}
\cosh((\alpha-1) s)^\frac{1}{\alpha-1},&\alpha\not=1\\
e^{c s},&\alpha=1, c>1.
\end{array}\right.
\end{equation}
 Moreover, 
\begin{enumerate}
\item If $\alpha> 1$, then   $\gamma$ has the two  branches asymptotic to a vertical translation of the lightlike cone $\mathcal{C}$.
\item If $\alpha=1$, then $\gamma$ is a pseudocircle $\s^1_1(r)$ centered at the origin or the two branches of $\gamma$ are asymptotic to $\mathcal{C}$.
\item If $0<\alpha<1$, then $\gamma$ meets tangentially $\mathcal{C}$ at two points.  
\item If $\alpha<0$, then $\gamma$ intersects orthogonally $\mathcal{C}$ at two points.
\end{enumerate}
\end{theorem}
 
\begin{proof}
Theorem  \ref{t31} and the computations before Thm. \ref{ti} imply that in hyperbolic polar coordinates, the function $\rho$ satisfies \eqref{k1}.   Now the timelike condition on $\gamma$ implies that if $\alpha=1$, then the constant $c$ satisfies $c>1$.  
  
  \end{proof}

\begin{figure}[h]
\begin{center}
\includegraphics[width=.3\textwidth]{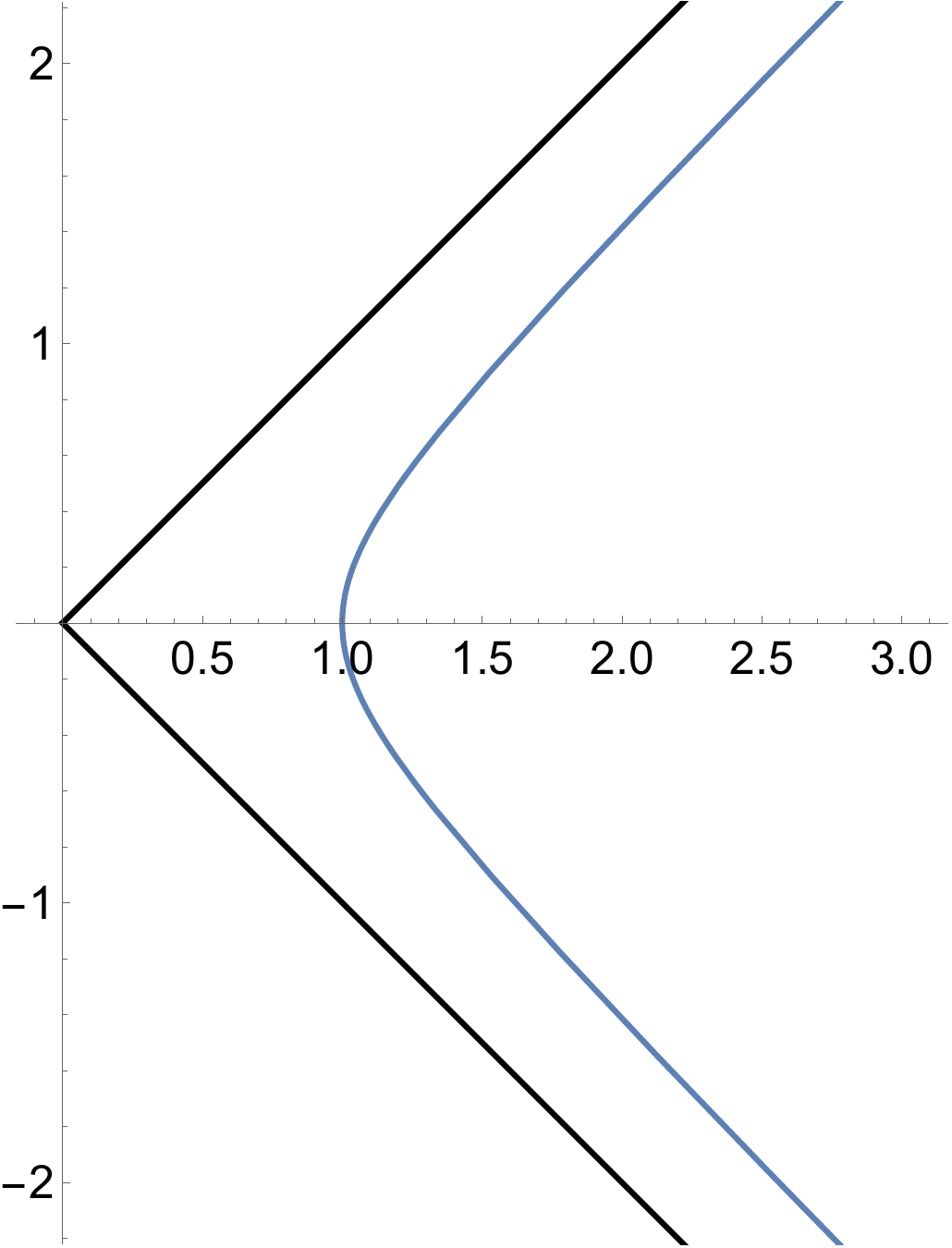}\,\includegraphics[width=.3\textwidth]{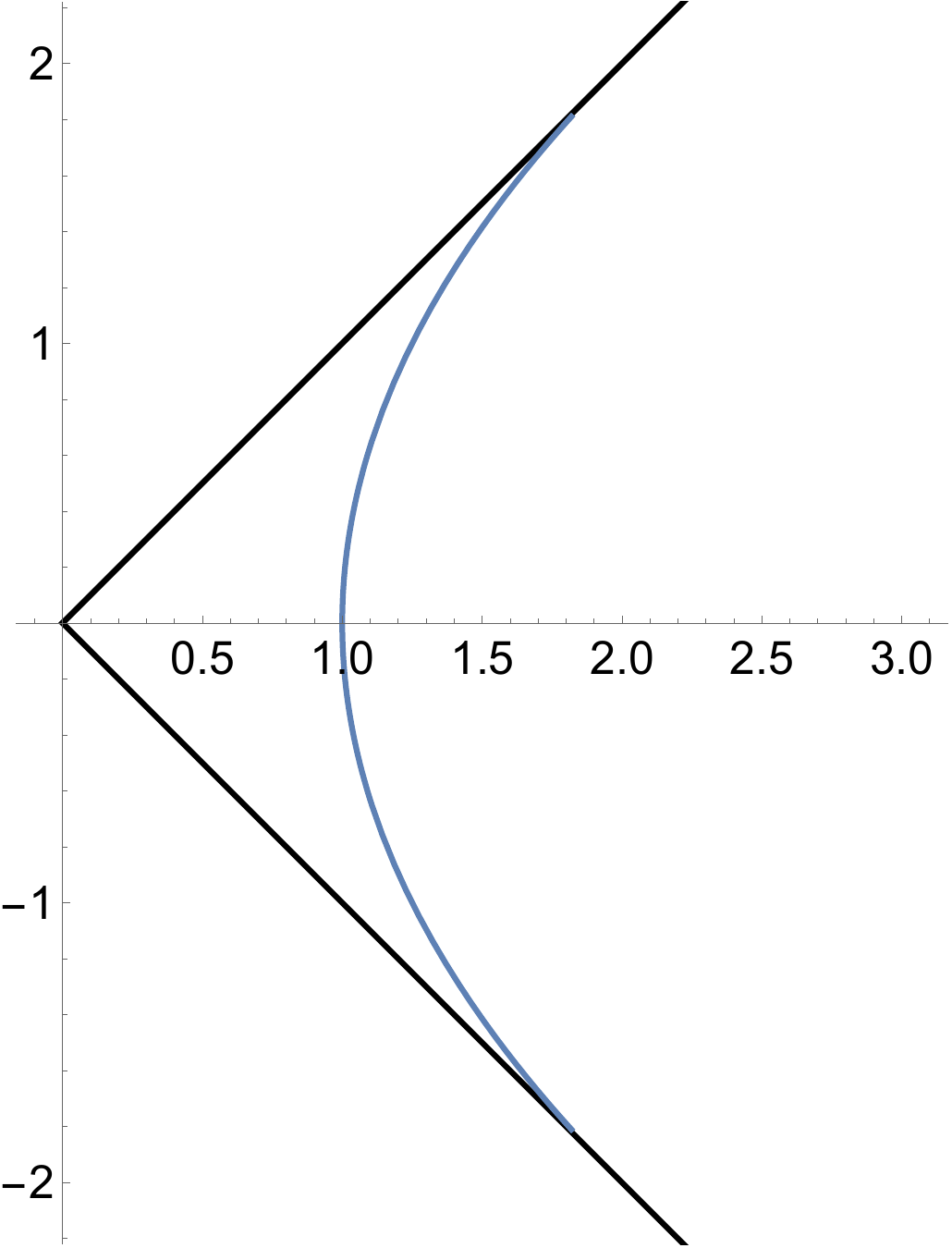}\,\includegraphics[width=.3\textwidth]{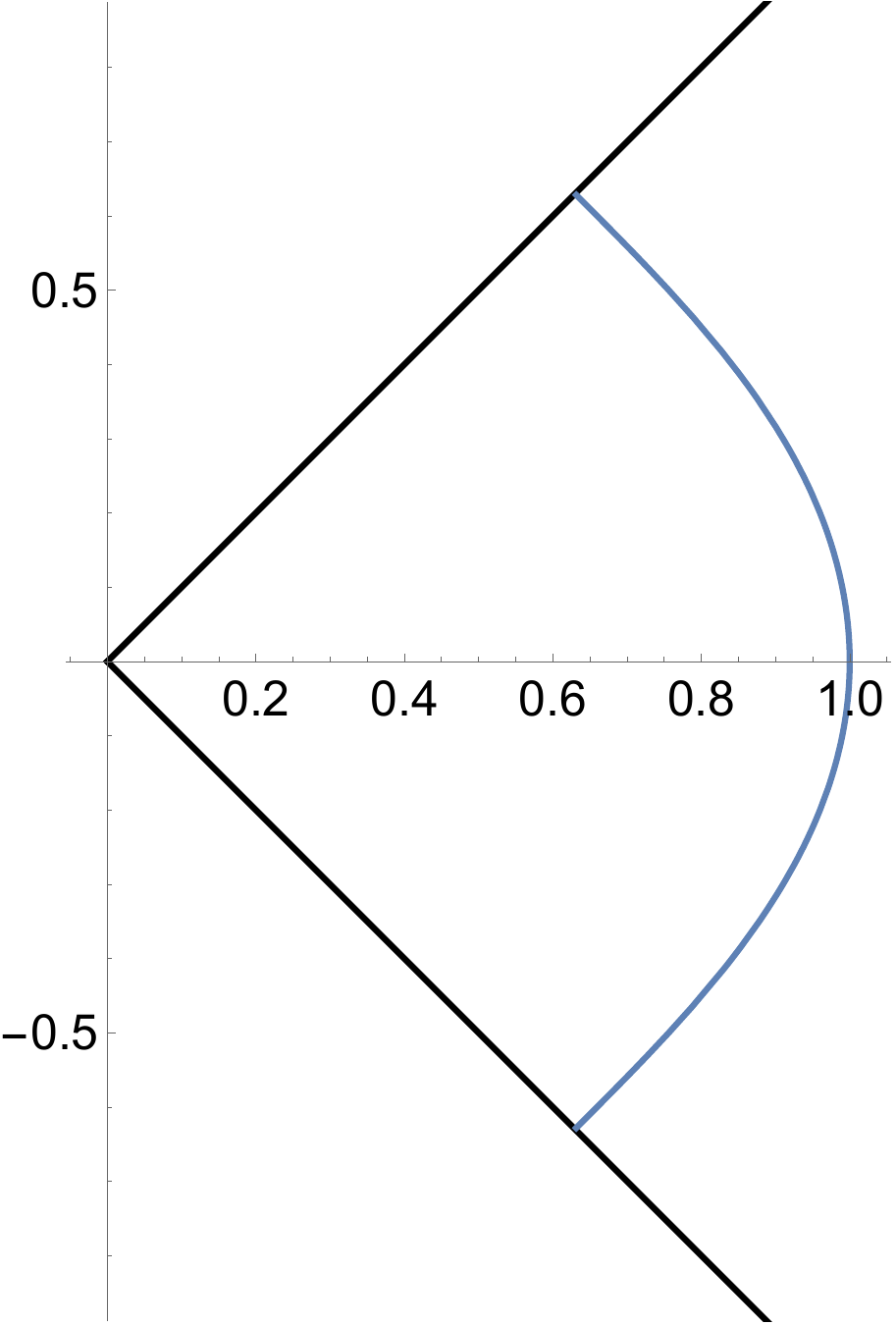}
\end{center}
\caption{ $\alpha$-stationary timelike curves contained in $\mathcal{C}^+$: $\alpha=2$ (left), $\alpha=\frac12$ (middle) and $\alpha=-2$ (right).  }
\label{fig4}
\end{figure}

Stationary timelike curves contained in $\mathcal{C}^-$ are obtained from Thm. \ref{pr42}: see Fig. \ref{fig5}. Now the function $\rho$ in hyperbolic polar coordinates satisfies Eq. \eqref{k2}.

\begin{theorem} \label{prt+}
Up to a linear isometry and a dilation of $\l^2$, the   $\alpha$-stationary timelike curves contained in $\mathcal{C}^-$ are parametrized   by 
\begin{equation}\label{so121}
\rho(s)=\left\{\begin{array}{ll}
\sinh((\alpha+1) s)^\frac{-1}{\alpha+1},&\alpha\not=-1\\
e^{c s},&\alpha=1, c^2>1.
\end{array}
\right.
\end{equation}
 Moreover, 
\begin{enumerate}
\item If $\alpha>0$, then a branch of  $\gamma$ touches $\mathcal{C}$  orthogonally and the other one is asymptotic to the $y$-axis. 
\item If $-1<\alpha<0$, then a branch of $\gamma$ goes to $\mathcal{C}^-$ tangentially.
\item If $\alpha=-1$, then a branch of $\gamma$ goes to $(0,0)$ and tangentially to   $\mathcal{C}$.  
\item If $\alpha<-1$, then a branch of $\gamma$ goes to $(0,0)$ tangentially to the $y$-axis and the other branch is asymptotic to a translation of the lightlike cone.
\end{enumerate}
 \end{theorem}

 \begin{figure}[h]
\begin{center}
\includegraphics[width=.2\textwidth]{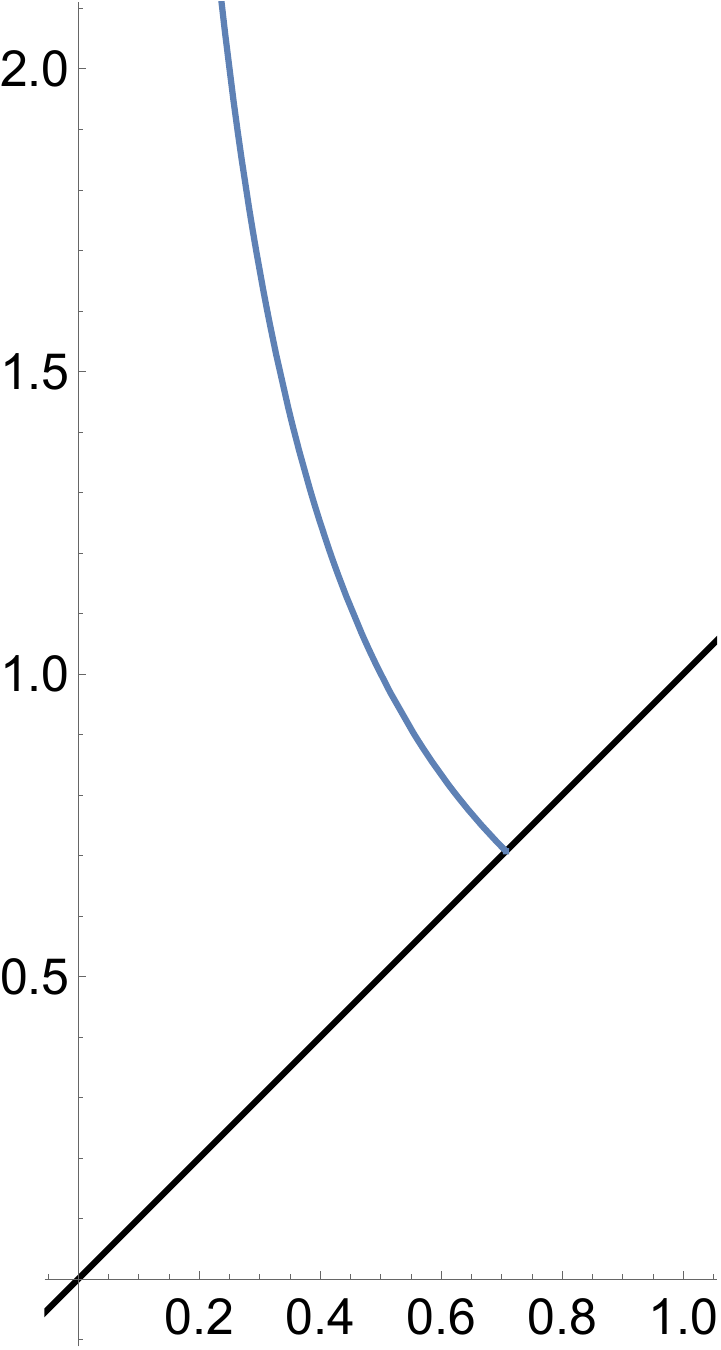}\,,\includegraphics[width=.3\textwidth]{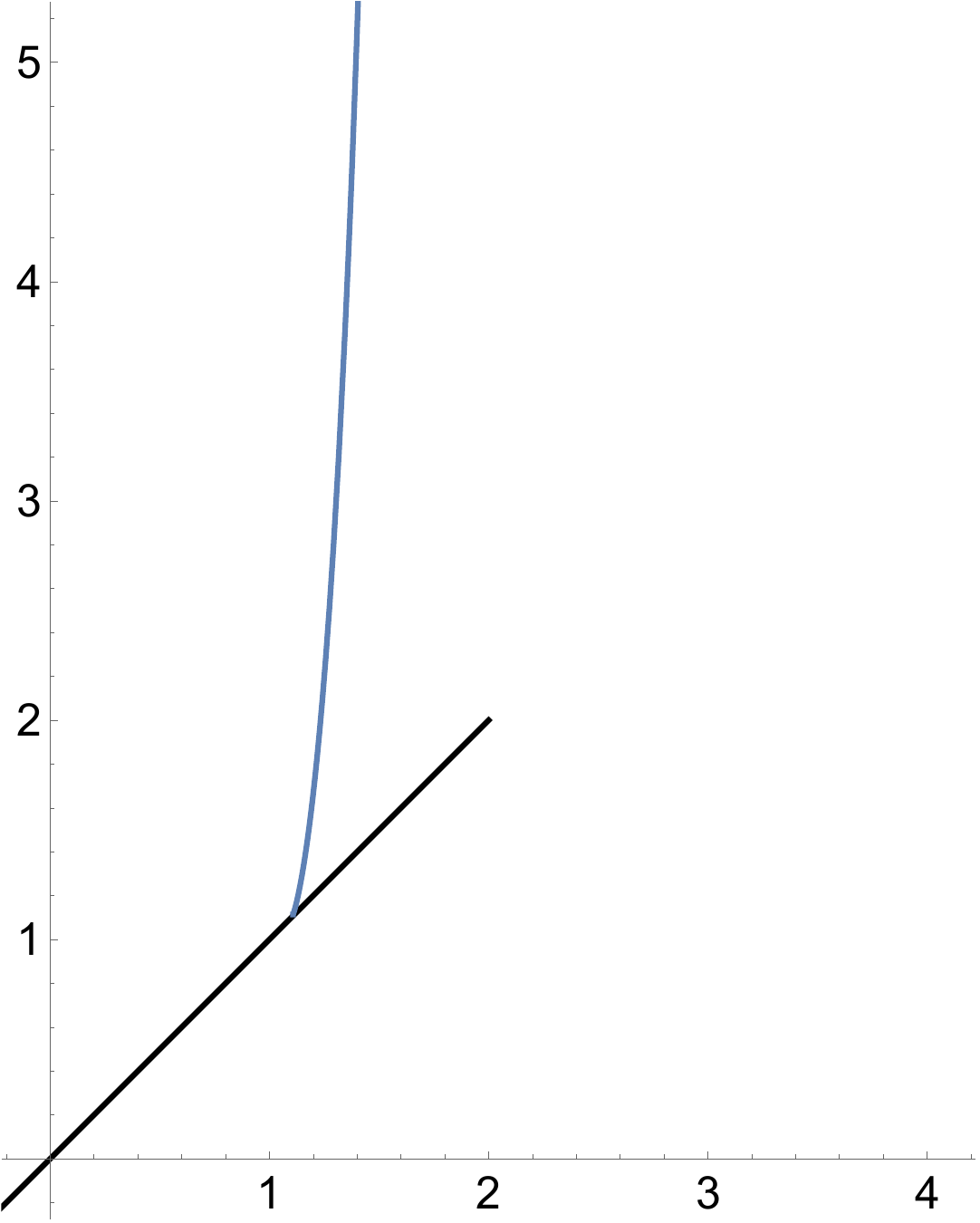}\,\includegraphics[width=.3\textwidth]{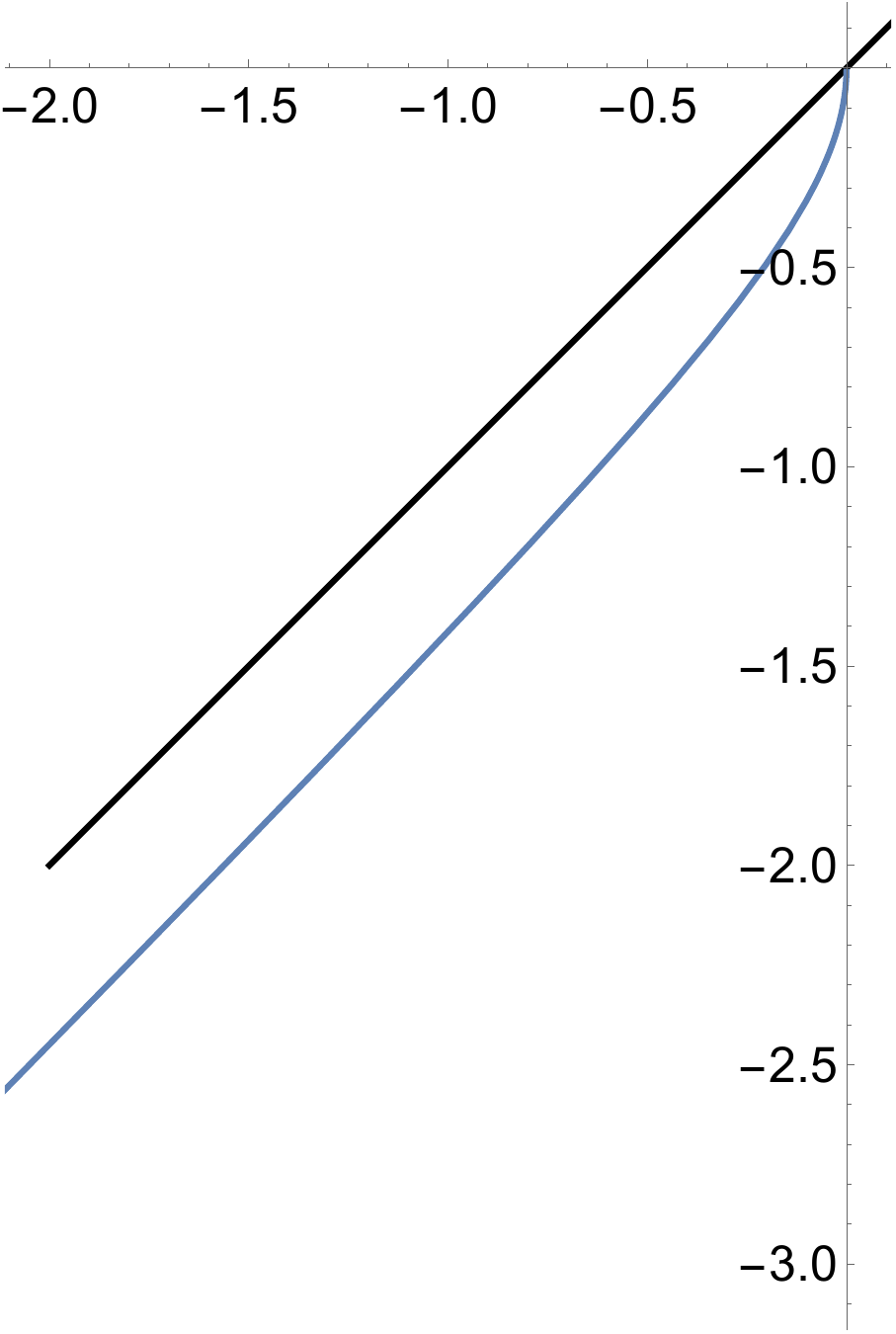}
\end{center}
\caption{  $\alpha$-stationary timelike curves contained in $\mathcal{C}^-$: $\alpha=1$ (left), $\alpha=-\frac18$ (middle) and $\alpha=-2$ (right).  }
\label{fig5}
\end{figure}

As we mentioned in Introduction, for spacelike curves, the energy \eqref{Eu1}, when $\alpha=0$, corresponds with the length of the curve. Critical points are straight lines. According to Thm. \ref{ti} the inverse curves of straight lines are $2$-stationary curves if the curve is included in $\mathcal{C}^-$ and $(-2)$-stationary curves if $\gamma$ is included in $\mathcal{C}^+$.

\begin{corollary} Up to a linear isometry and a dilation of $\l^2$, the inverse curves of spacelike straight lines are:
\begin{enumerate}
\item The curve $\gamma(s)=(\cosh(s)\sinh(s),\cosh(s)^2)$, if it is contained in $\mathcal{C}^-$. This is a $2$-stationary curve.  
\item The curve $\gamma(s)=(-\sinh(s)\cosh(s),-\sinh(s)^2)$, if it is contained in $\mathcal{C}^+$. This is a $(-2)$-stationary curve. 
\end{enumerate}
\end{corollary}

 As a consequence of the  classification of stationary curves, there are non-degenerate curves in $\l^2$ with mix causal character satisfying Eq. \eqref{eq1} for the same value of $\alpha$. These curves orthogonally cross the lightlike cone, as shown in the following example.
 
 \begin{example} Let  $\alpha=-2$. Consider $\gamma^+=\gamma^+(s)$ and $\gamma^{-}=\gamma^{-}(s)$ to be the $(-2)$-stationary spacelike and timelike curves  given in \eqref{so} and \eqref{sot2}, respectively. Both curves tend to $\mathcal{C}$ orthogonally at finite points which coincide for both curves: see Figs. \ref{fig2} and \ref{fig4}, respectively. Then we can end $\gamma^{-}$ at $\mathcal{C}$ and continue with $\gamma^{+}$. For this purpose, we reparametrize the domains of both curves changing  the domain $(-\infty,\infty)$ of $\gamma^+$  by $(-\infty,0)$ and the  domain $(-\infty,\infty)$   of $\gamma^{-}$  by $(0,\infty)$. Then we have
$$\gamma(s)=\left\{\begin{array}{ll}
\gamma^+(-\log(-s)),&s\in (-\infty,0)\\
( 2^\frac{-2}{3},2^\frac{-2}{3}),& s=0\\
\gamma^-(-\log(s)),&s\in (0,\infty).
\end{array}\right.$$ 
In addition, if we   take their symmetries about the $x$-axis and the $y$-axis, we obtain a closed curve with mix causal character and satisfying Eq.  \eqref{eq1} for the same value of $\alpha=-2$. See Fig. \ref{fig6}, left.

A similar process can be done when $\alpha>0$ and for $\alpha$-stationary spacelike curves contained in $\mathcal{C}^+$ and $\alpha$-stationary timelike curves contained in $\mathcal{C}^-$. Again, we obtain a curve with mix causal character satisfying Eq. \eqref{eq1} for the same value $\alpha$. This curve is tangent to the $x$-axis and the $y$-axis. See Fig. \ref{fig6}, right.
\end{example}  
\begin{figure}[h]
\begin{center}
\includegraphics[width=.4\textwidth]{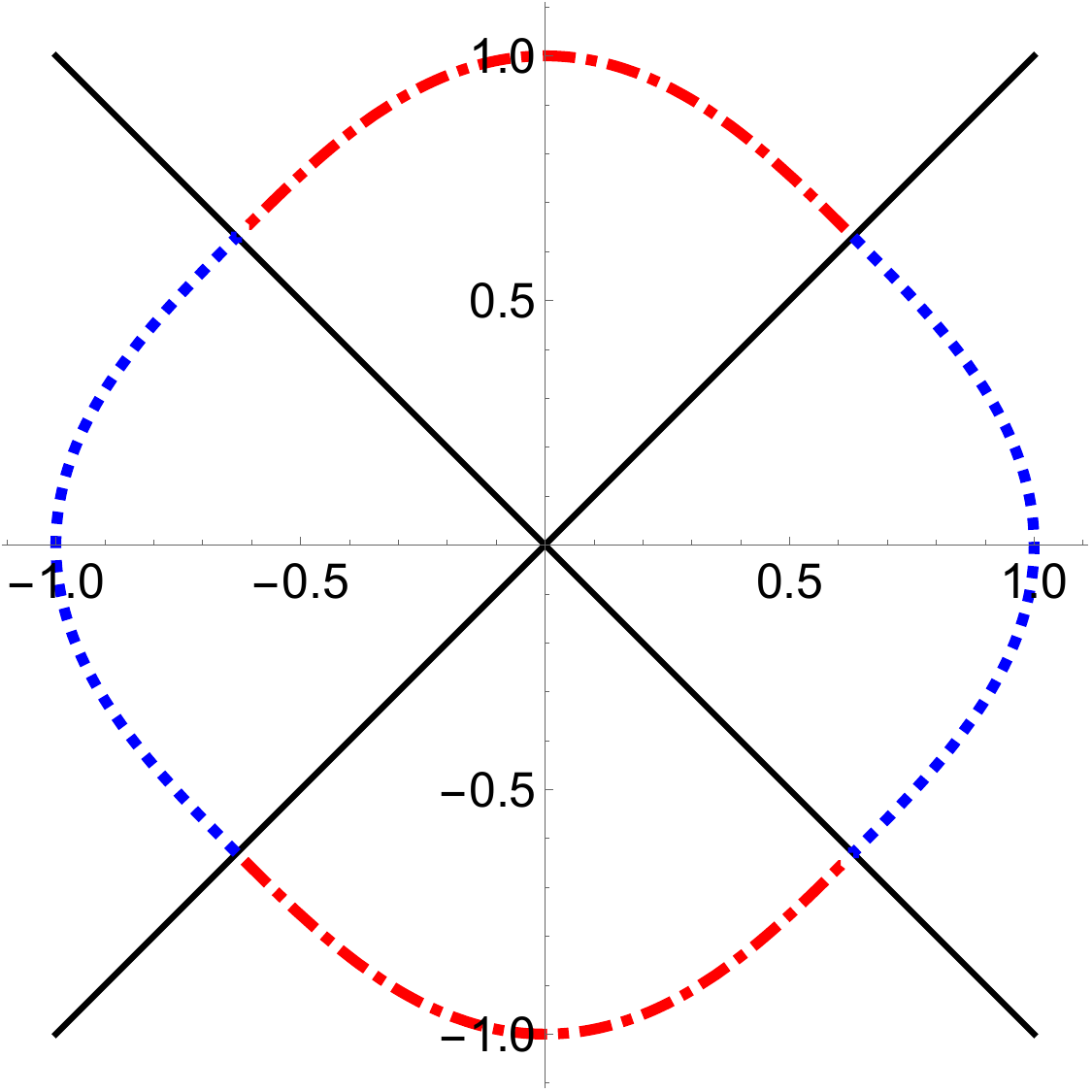}\qquad \includegraphics[width=.4\textwidth]{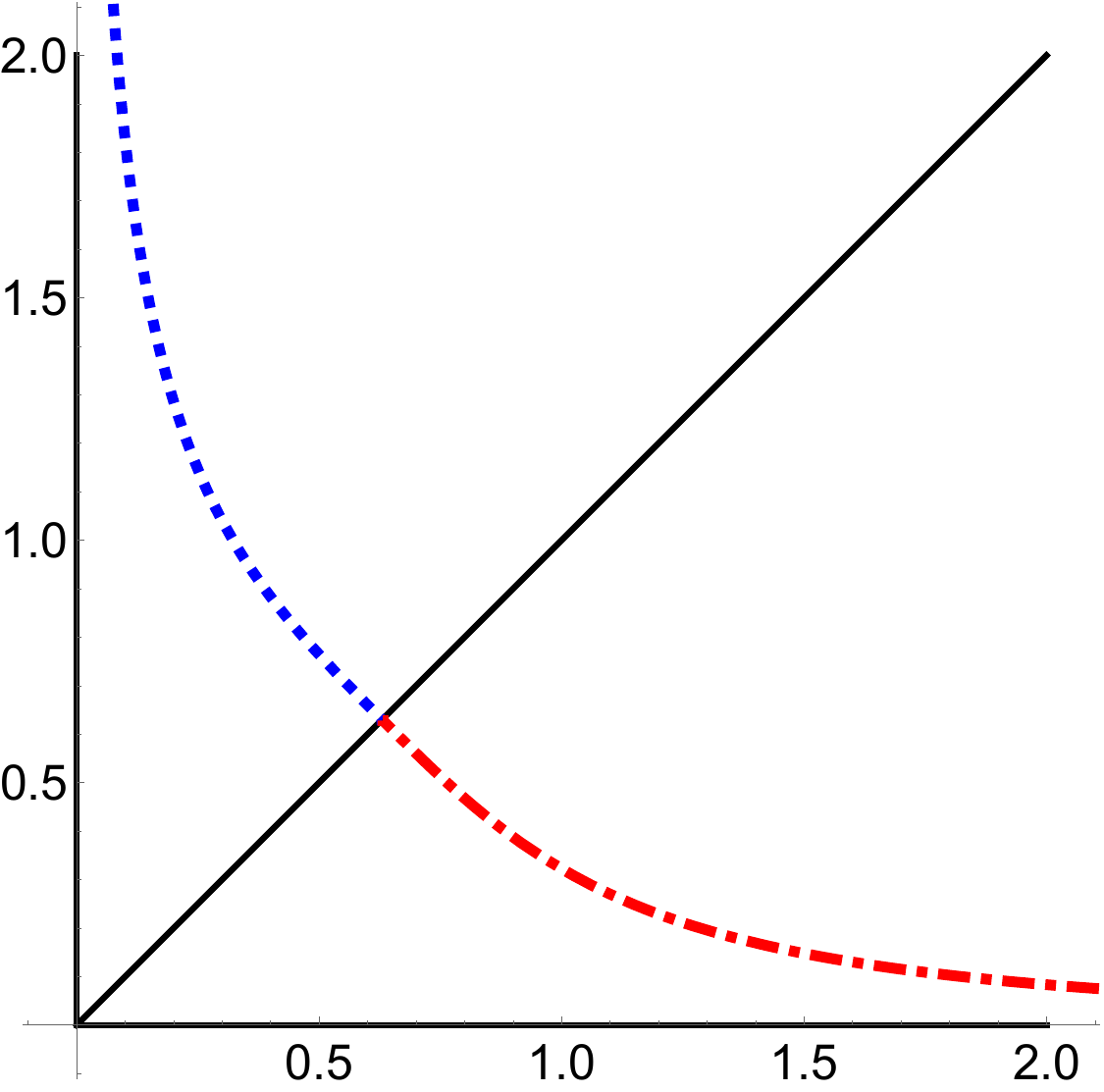}
\end{center}
\caption{Stationary curves with mix causal character: $\alpha=-2$ (left) and $\alpha=2$ (right).    The red dashed curve  is spacelike and the blue dotted curve is timelike.  }
\label{fig6}
\end{figure}
  \section{The problem of maximizing the energy}\label{s5}
  
We go back to the initial Euler's problem in $\l^2$:  given two points $P_1, P_2\in\l^2$, find a spacelike curve $\gamma$ which maximizes the energy $E_\alpha$. As usually, the first approach  is finding the critical points of this energy joining both points, which are potential maximizers of $E_\alpha$. In $\r^2$, the problem of finding minimizers of $E_\alpha$ was solved by Mason for the value $\alpha=2$ \cite{ma} and a general approach for all values of $\alpha$ has been recently obtained in \cite{dl1}. In these results, techniques as the  extremal fields  in the sense of Weierstrass are employed \cite{gh}.  
  
  A simple but illustrative case is when $P_1$ and $P_2$ are collinear with the origin $0=(0,0)$. A possible maximizer is the straight line joining both points. Since we focus on the spacelike case, this line is assumed to be spacelike. Thus,   both points must be contained in $\mathcal{C}^+$; otherwise, the line joining $P_1$ and $P_2$ would meet  $\mathcal{C}$, which we want to avoid. Let $S_{12}$ denote the segment joining $P_1$ and $P_2$. Under this situation, we prove the following result.

\begin{theorem} \label{t52}
Let $\alpha\in\r$. Let $P_1$ and $P_2$ be two given points contained in $\mathcal{C}^+$. If $P_1$ and $P_2$ are collinear with $0$ and $0\not\in S_{12}$, then the maximizer of the energy $E_\alpha$ is the segment $S_{12}$.
\end{theorem}

\begin{proof} Without loss of generality, we may assume that $P_1$ and $P_2$ lie in the component of $\mathcal{C}^+$ included in $x>0$. We use hyperbolic polar coordinates, $(\rho,\theta)\equiv \rho(\cosh\theta,\sinh\theta)$. For a spacelike curve $\gamma(s)=(\rho(s),\theta(s))$ contained in $\mathcal{C}^+$, we have 
$E_\alpha[\gamma]=\int_a^b\rho^{\alpha}\sqrt{\rho'^2-\rho^2\theta'^2}\, ds$. 
 
Up to a linear isometry, a dilation of $\l^2$ and renaming the points, we can assume $P_1=(1,0)$ and $P_2=(r,0)$  with $r>1$. First, we find the value of $E_\alpha[S_{12}]$. We parametrize $S_{12}$ in hyperbolic polar coordinates  by $s\mapsto (\rho(s),\theta(s))$, where $\rho(s)=s$, $\theta(s)=0$ and $s\in [1,r]$.   Then
$$E_\alpha[S_{12}]=\int_1^r s^\alpha\, ds=\left\{\begin{array}{ll}r^{\alpha+1}-1,&\alpha\not=-1\\ \log(r)&\alpha=-1.\end{array}\right.$$
Let $\gamma$ be any spacelike curve joining $P_1$ and $P_2$, $s\in [a,b]$. Using $\theta'^2\geq 0$, we have   
\begin{equation*}
E_\alpha[\gamma]=\int_a^b\rho^\alpha\sqrt{\rho'^2-\rho^2\theta'^2}\, ds\leq \int_a^b\rho^\alpha \rho' \, ds=
E_\alpha[S_{12}].
\end{equation*}

\end{proof}

We now consider the case that $P_1$ and $P_2$ are collinear with $0$ and $0\in S_{12}$.  

\begin{theorem}  
Let $\alpha\in\r$. Let $P_1$ and $P_2$ be two given points contained in $\mathcal{C}^+$ which are collinear with $0$ and $0\in S_{12}$. 
\begin{enumerate}
\item If $\alpha>-1$, the maximizer of the energy $E_\alpha$ is the segment $S_{12}$.
\item If $\alpha\leq -1$, there is not a maximizer of the energy $E_\alpha$ joining both points.
\end{enumerate}
\end{theorem}

\begin{proof} Without loss of generality, we suppose that in hyperbolic polar coordinates, $P_1=(1,0)$ and $P_2=(-r,0)$, $r>0$. 
\begin{enumerate}
\item Case $\alpha>-1$. Let $\gamma=\gamma(s)$, $s\in  [a,b]$, be any spacelike curve joining $P_1$ and $P_2$. Choose a point $Q=\gamma(c)$ such that the lengths of  $\gamma_{|[a,c]}$  and $\gamma_{|[c,b]}$ are both less than those of $S_{10}$ and $S_{02}$, respectively, where $S_{10}$ and $S_{02}$ are the segments joining $P_1$ and $0$ and $0$ and $P_2$.  Then $E_\alpha[\gamma_{|_{[a,c]}}]\leq E_\alpha[S_{10}]$ and $E_\alpha[\gamma_{|_{[c,b]}}]\leq E_\alpha[S_{02}]$. Notice that $E_\alpha[S_{10}]$ and $E_\alpha[S_{02}]$ are finite because $\alpha>-1$. Thus $E_\alpha[\gamma]\leq  E_\alpha[S_{10}]+E_\alpha[S_{02}]=E_\alpha[S_{12}]$.
\item Case $\alpha\leq -1$. It is immediate that $E_\alpha[S_{12}]=\infty$, proving the non-existence of maximizers.
\end{enumerate}
  
  \end{proof}
\section*{Acknowledgement}

 Rafael L\'opez   has been partially supported by MINECO/ MICINN/FEDER grant no. PID2023-150727NB-I00, and by the ``Mar\'{\i}a de Maeztu'' Excellence Unit IMAG, reference CEX2020-001105- M, funded by MCINN/AEI/ 10.13039/501100011033/ CEX2020-001105-M.


\end{document}